\documentclass[prodmode,acmtoms]{acmtrans2m}

\makeatletter
\def\ps@myheadings{\let\@mkboth\@gobbletwo
\def\@oddhead{\hbox{}\hfill \small\sf \rightmark\hskip 19pt{\Large$\cdot$}\hskip 17pt\mypage}
\def\@oddfoot{\hbox{}\hfill\tiny\@runningfoot}
\def\@evenhead{\small\sf\mypage \hskip 17pt{\Large$\cdot$}\hskip 19pt\leftmark\hfill \hbox{}}
\def\@evenfoot{\tiny\@runningfoot\hfill\hbox{}}
\def\sectionmark##1{}\def\subsectionmark##1{}}
\def\@runningfoot{}
\def\runningfoot{\def\@runningfoot{}}
\def\@firstfoot{}
\def\firstfoot{\def\@firstfoot{}}
\def\ps@titlepage{\let\@mkboth\@gobbletwo
\def\@oddhead{}\def\@oddfoot{\hbox{}\hfill
\tiny\@firstfoot}\def\@evenhead{}\def\@evenfoot{\tiny\@firstfoot\hfill\hbox{}}}
\makeatother

\usepackage{url,etex,numprint}
\usepackage[utf8x]{inputenc}
\usepackage{epstopdf}

\npthousandsep{\,}
\npdecimalsign{.}
\npproductsign{\cdot}
\npunitseparator{\,}

\usepackage{graphics,graphicx}
\graphicspath{{./figure/}}

\usepackage{xcolor,float}
\usepackage{amsfonts}
\usepackage[boxruled,longend,linesnumbered]{algorithm2e}

\newtheorem{problem}{Problem}
\usepackage[english]{babel}

\usepackage{graphics,graphicx}
\usepackage{makeidx}
\usepackage{bm,bbm,cases}

\usepackage[math]{easyeqn}
\usepackage{easyvector}

\eqrowsep{2pt plus 1pt minus 1pt}
\eqspacing{4pt plus 2pt minus 1pt}

\newvector[\mathit{0},\bm{0}]{zero}
\newvector[\mathit{a},\bm{a}]{avec}
\newvector[\mathit{b},\bm{b}]{bvec}
\newvector[\mathit{c},\bm{c}]{cvec}
\newvector[\mathit{d},\bm{d}]{dvec}
\newvector[\mathit{p},\bm{p}]{pvec}

\newcommand{\ds}{\,\mathrm{d}s}

\newcommand{\dxi}{\,\mathrm{d}\xi}

\newcommand{\dz}{\,\mathrm{d}z}

\newcommand{\dtau}{\,\mathrm{d}\tau}
\newcommand{\du}{\,\mathrm{d}u}

\newcommand{\SIGN}{\textrm{sign}}

\newcommand{\Cf}{\mathcal{C}}
\newcommand{\Sf}{\mathcal{S}}

\newcommand{\reffun}[1]{\texttt{\ref{#1}}}
\IncMargin{0.5em}

\newcommand{\DEF}{\mathrel{\mathop:}=}
\newcommand{\assign}{\leftarrow}

\newcommand{\ASSIGNl}[1]{\rlap{$#1$}\qquad\assign\,}
\newcommand{\ASSIGNm}[1]{\rlap{$#1$}\quad\;\assign\,}
\newcommand{\ASSIGNs}[1]{\rlap{$#1$}\quad\assign\,}

\newtheorem{theorem}{Theorem}[section]

\newtheorem{lemma}[theorem]{Lemma}
\newtheorem{corollary}[theorem]{Corollary}
\newtheorem{assumption}[theorem]{Assumption}

\newdef{remark}[theorem]{Remark}

\newtheorem{definition}[theorem]{Definition}

\begin{document}

\markboth{E.Bertolazzi and M.Frego}{Fast and accurate clothoid fitting}

\title{Fast and accurate $G^1$ fitting of clothoid curves}

\author{
\uppercase{Enrico Bertolazzi} and \uppercase{Marco Frego}
\\
University of Trento, Italy
}

\begin{abstract}
A new effective solution to the problem of Hermite $G^1$ interpolation 
with a clothoid curve is here proposed, that is a clothoid that interpolates two given points in a plane with assigned unit tangent vectors.
The interpolation problem is a system of three nonlinear equations
with multiple solutions which is difficult to solve also numerically.
Here the solution of this system is reduced to the computation of 
the zeros of one single function in one variable.
The location of the zero associated to the relevant solution 
is studied analytically: the interval containing the zero where
the solution is proved to exists and to be unique is provided.
A simple guess function allows to find that zero with very few 
iterations in all possible configurations.

The computation of the clothoid curves and the solution algorithm 
call for the evaluation of Fresnel related integrals.
Such integrals need asymptotic expansions near critical values to avoid loss of precision.
This is necessary when, for example, the solution of the interpolation problem 
is close to a straight line or an arc of circle.
A simple algorithm is presented for efficient computation
of the asymptotic expansion.

The reduction of the problem to a single nonlinear function in one variable
and the use of asymptotic expansions make the present solution algorithm fast and robust.
In particular a comparison with algorithms present in literature shows that 
the present algorithm requires less iterations.
Moreover accuracy is maintained in all possible configurations 
while other algorithms have a loss of accuracy near the transition zones.

\end{abstract}






\maketitle

\section{Introduction}

There are several curves proposed
for Computer Aided Design either ~\cite{Farin:2001,Baran:2010,deBoor:1978},
for trajectories planning of robots and vehicles 
and for geometric roads layout \cite{dececco:2007,Scheuer:1997}.

The most important curves are the clothoids (also known as Euler's or Cornu's spirals), the clothoids splines (a planar curve consisting in clothoid segments), circles and straight lines,\cite{Davis:1999,Meek:1992,Meek:2004,Meek:2009,Walton:2009},
the generalized clothoids or Bezier spirals \cite{Walton:1996}.
Pythagorean Hodograph \cite{Walton:2007,Farouki:1995}, bi-arcs and conic curves are also widely 
used~\cite{Pavlidis:1983}.
It is well known that clothoids are extremely useful despite their transcendental form.

The procedure that allows a curve to interpolate two given points in a plane 
with assigned unit tangent vectors is called $G^1$ Hermite interpolation, while if the curvatures are given at the two points,
then this is called $G^2$ Hermite interpolation \cite{McCrae:2008}.
A single clothoid segment is not enough to ensure $G^2$ Hermite interpolation,
because of the insufficient degrees of freedom. For example
the interpolation problem is solved by using composite clothoid segments in~\cite{Shin:1990}
and using cubic spirals in \cite{Kanayama:1989}.
However, in some applications is enough the cost-effectiveness of a $G^1$
Hermite interpolation \cite{Walton:2009,Bertolazzi2006} especially when the discontinuity of the curvature
is acceptable.\\
The purpose of this paper is to describe a new method for
$G^1$ Hermite interpolation with a single clothoid segment; this method,
which does not require to split the problem in mutually exclusive 
cases (as in \citeNP{Meek:2009}), is fully effective even in case 
of Hermite data like straight lines or circles (see Figure~\ref{fig:1and2} - right). 
These are of course limiting cases but are naturally treated in the present approach.
In previous works the limiting cases are treated separately introducing thresholds. 
The decomposition in mutually exclusive states is just a geometrical fact that helps to understand the problem
but introduces instabilities and inaccuracies which are
absent in the present formulation, as we show in the section of numerical tests.

Finally, the problem of the $G^1$ Hermite interpolation 
is reduced to the computation of a zero of a unique nonlinear equation.
The Newton--Rasphson iterative algorithm is used to accurately compute
the zero and a good initial guess is also derived 
so that few iterations (less than four) suffice.

The article is structured as follows. 
In section~\ref{sec:2} we define the interpolation problem, in section~\ref{sec:3} we describe the passages to reformulate it such that from three equations in three unknowns  it reduces to one nonlinear equation in one unknown. Section~\ref{sec:existence} describes the theoretical aspects of the present algorithm. Here it is proved the existence of the solution and how to select a valid solution among the infinite possibilities. A bounded range where this solution exists and is unique is provided.
Section~\ref{sec:5} is devoted to the discussion of a good starting point for the Newton--Raphson method, so that using that guess, quick convergence is achieved. Section \ref{sec:6} introduces the Fresnel related integrals, e.g. the momenta of the Fresnel integrals. Section \ref{sec:6} analyses the stability of the computation of the clothoid and \textit{ad hoc} expressions for
critical cases are provided. 
Section~\ref{sec:7} is devoted to numerical tests and comparisons with other methods present in literature.
In the Appendix a pseudo-code complete the presented algorithm
for the accurate computation of the Fresnel related integrals.

\section{The fitting problem}\label{sec:2}
Consider the curve which satisfies the following differential equations:
\begin{EQ}[rclrcl]\label{eq:ODE:clot}
  x'(s)         &=& \cos \vartheta(s),     \qquad & x(0)&=&x_0,\\
  y'(s)         &=& \sin \vartheta(s),     \qquad & y(0)&=&y_0, \\
  \vartheta'(s) &=& \mathcal{K}(s), \qquad & \vartheta(0)&=&\vartheta_0, \\
\end{EQ}
where $s$ is the arc parameter of the curve, $\vartheta(s)$ is the direction of 
the tangent $(x'(s),y'(s))$ and $\mathcal{K}(s)$ is the 
curvature at the point $(x(s),y(s))$.
When $\mathcal{K}(s)\DEF\kappa' s + \kappa$, i.e. when the curvature changes linearly,
the curve is called  Clothoid.
As a special case, when $\kappa'=0$ the curve has constant curvature, i.e. is a circle
and when both $\kappa=\kappa'=0$ the curve is a straight line.
The solution of ODE~\eqref{eq:ODE:clot} is given in the next definition:
\begin{definition}[Clothoid curve]
The general parametric form of a clothoid spiral curve is the following
\begin{EQ}[rcl]\label{clot}
  x(s) &=& x_0 + \int_0^s \cos \Big(\frac{1}{2}\kappa'\tau^2+\kappa\tau+\vartheta_0\Big) \dtau, \\
  y(s) &=& y_0 + \int_0^s \sin \Big(\frac{1}{2}\kappa'\tau^2+\kappa\tau+\vartheta_0\Big) \dtau.
\end{EQ}
Notice that $\frac{1}{2}\kappa's^2+\kappa s+\vartheta_0$ and 
$\kappa's+\kappa$ are, respectively, the angle and the curvature
at the  abscissa $s$.
\end{definition}
The computation of the integrals~\eqref{clot} is here recast as a combination
(discussed in Section \ref{sec:6}) of Fresnel sine $\Sf(t)$ and cosine $\Cf(t)$
functions. Among the various possible definitions, we choose the following one,  \citeNP{abramowitz:1964}.
\begin{definition}[Fresnel integral functions]
\begin{EQ}\label{fresnel}
  \Cf(t)=\int_0^t\cos \left(\frac{\pi}{2}\tau^2\right)\dtau,\qquad
  \Sf(t)=\int_0^t\sin \left(\frac{\pi}{2}\tau^2\right)\dtau.\qquad
\end{EQ} 
\end{definition}

\begin{remark}
The literature reports different definitions, such as:
\begin{EQ}
  \widetilde\Cf(t) = \int_0^t\cos(\tau^2)\dtau,
  \qquad
  \widetilde\Sf(t) = \int_0^t\sin(\tau^2)\dtau,
\end{EQ}
or
\begin{EQ}\label{eq:SC:alternative}
  \widehat\Cf(\theta) = \dfrac{1}{\sqrt{2\pi}}\int_0^\theta\dfrac{\cos u}{\sqrt{u}}\du,
  \qquad
  \widehat\Sf(\theta) = \dfrac{1}{\sqrt{2\pi}}\int_0^\theta\dfrac{\sin u}{\sqrt{u}}\du.
\end{EQ}
The identities \eqref{eq:equiv} allow to switch among these definitions:
\begin{EQ}[rclcl]\label{eq:equiv}
  \Cf(t)
  &=&
  \int_0^{\frac{\sqrt{2}}{\sqrt{\pi}}t}\cos\left(\tau^2\right)\dtau
  &=&
  \dfrac{\SIGN(t)}{\sqrt{2\pi}}\int_0^{\frac{\pi}{2}t^2}
  \dfrac{\cos u}{\sqrt{u}}\du
  \\
  \Sf(t)
  &=&
  \int_0^{\frac{\sqrt{2}}{\sqrt{\pi}}t}\sin\left(\tau^2\right)\dtau
  &=&
  \dfrac{\SIGN(t)}{\sqrt{2\pi}}\int_0^{\frac{\pi}{2}t^2}
  \dfrac{\sin u}{\sqrt{u}}\du.
\end{EQ}
\end{remark}

Thus, the problem considered in this paper is stated next.
\begin{problem}[Clothoid Hermite interpolation]\label{prob:1}
  Given two points $(x_0,y_0)$ and $(x_1,y_1)$
  and two angles $\vartheta_0$ and $\vartheta_1$, 
  find a clothoid segment of the form~\eqref{clot} which
  satisfies:
  \begin{EQ}[rclrclrcl]\label{eq:prob:1}
    x(0) &=& x_0, \qquad &
    y(0) &=& y_0, \qquad &
    (x'(0)\,,\,y'(0)) &=& (\cos\vartheta_0,\,\sin\vartheta_0), \\[\jot]
    x(L) &=& x_1, \qquad &
    y(L) &=& y_1, \qquad &
    (x'(L),y'(L)) &=& (\cos\vartheta_1,\,\sin\vartheta_1),
  \end{EQ}
  where $L>0$ is the length of the curve segment.
\end{problem}
The general scheme is showed in Figure~\ref{fig:1and2} - left.
\begin{remark}
  Notice that Problem \ref{prob:1} admits an \emph{infinite} number of solutions.
  In fact, given $\vartheta(s)$, the angle of a clothoid which solves 
  Problem~\ref{prob:1}, satisfies $\vartheta(0)=\vartheta_0+2 k\pi$ and
  $\vartheta(L)=\vartheta_1+2\ell\pi$ with $k,\ell\in\mathbb{Z}$: different values of $k$ 
  correspond to different interpolant curves that loop around the initial and the final point.
  Figure~\ref{fig:1and2}- right shows possible solutions
  derived from the same Hermite data.
\end{remark}
\begin{figure}[!bt]
  \begin{center}
    \includegraphics[scale=0.75]{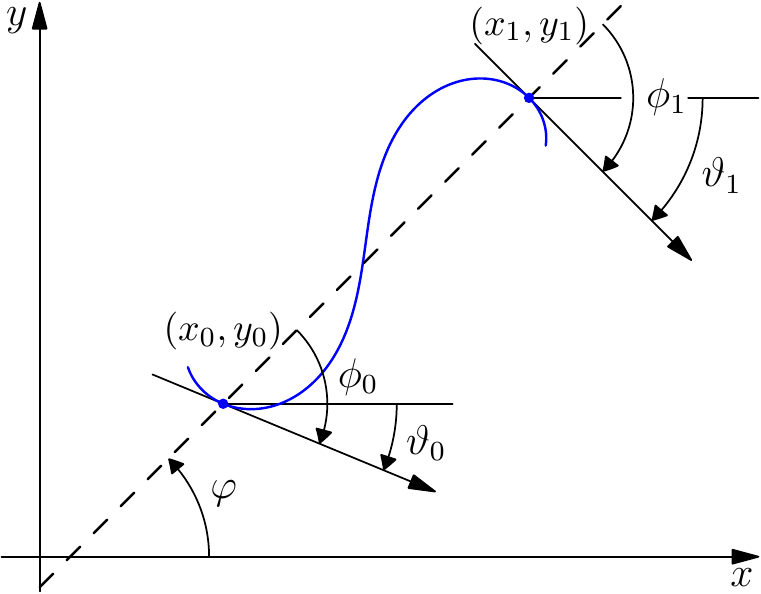}
    \qquad
    \includegraphics[scale=0.75]{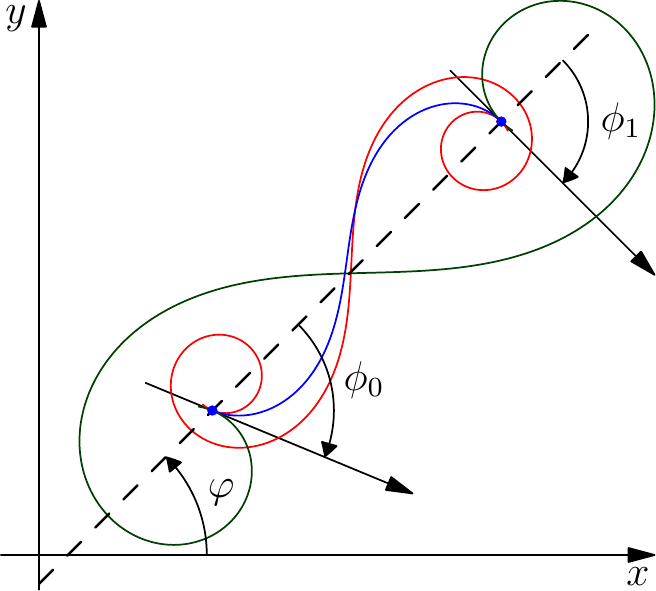}
  \end{center}
  \caption{Left: $G^1$ Hermite interpolation schema and  notation. Right: some possible solutions.}
  \label{fig:1and2}
\end{figure}
The solution of Problem~\ref{prob:1} is a zero of the following
nonlinear system involving the unknowns $L$, $\kappa$, $\kappa'$:

\begin{EQ}\label{eq:F:nonlin}
   \bm{F}(L,\kappa, \kappa')
   =
   \pmatrix{
     x_1-x_0 -\int_0^L \cos \left(\frac{1}{2}\kappa's^2+\kappa s+\vartheta_0\right) \ds\\
     y_1-y_0 -\int_0^L \sin \left(\frac{1}{2}\kappa's^2+\kappa s+\vartheta_0\right) \ds\\
     \vartheta_1-\left(\frac{1}{2}\kappa'L^2+\kappa L+\vartheta_0\right)
   }.\qquad
\end{EQ}
In Section~\ref{sec:3} the nonlinear system \eqref{eq:F:nonlin}
is reduced to a single non linear equation that 
allows for the efficient solution Problem~\ref{prob:1}.

\section{Recasting the interpolation problem}
\label{sec:3}
The nonlinear system~\eqref{eq:F:nonlin} is recasted in an equivalent 
form by introducing the parametrization $s=\tau L$ so that
the involved integrals have extrema which do not dependent on $L$:
\begin{EQ}[rcl]\label{eq:F:nonlin:2}
   \bm{F}\left(L,\dfrac{B}{L},\dfrac{2\,A}{L^2}\right)
   &=&
   \pmatrix{
     \Delta x -L\int_0^1 \cos \left(A\tau^2+B\tau+\vartheta_0\right) \dtau \\
     \Delta y -L\int_0^1 \sin \left(A\tau^2+B\tau+\vartheta_0\right) \dtau \\
     \vartheta_1-\left(A+B+\vartheta_0\right)
   },\qquad
\end{EQ}
where $A=\frac{1}{2}\kappa' L^2$, $B=L\,\kappa$, $\Delta x = x_1-x_0$, $\Delta y = y_1-y_0$.
It is convenient to introduce the following functions whose properties 
are studied in Section~\ref{sec:6}:
\begin{EQ}[rcl]\label{eq:XY0}
  X(a,b,c)&=&\int_0^1 \cos\left(\frac{a}{2}\tau^2+b\tau+c\right)\dtau,\\
  Y(a,b,c)&=&\int_0^1 \,\sin\left(\frac{a}{2}\tau^2+b\tau+c\right)\dtau.
\end{EQ}
Combining~\eqref{eq:XY0} with~\eqref{eq:F:nonlin:2} becomes:
\begin{EQ}[rcl]\label{eq:F:nonlin:2:b}
   \bm{F}\left(L,\dfrac{B}{L},\dfrac{2\,A}{L^2}\right)
   &=&
   \pmatrix{
     \Delta x -L\cdot X(2A,B,\vartheta_0) \\
     \Delta y -L\cdot Y(2A,B,\vartheta_0) \\
     \vartheta_1-\left(A+B+\vartheta_0\right)
   }.\qquad
\end{EQ}
The third equation in \eqref{eq:F:nonlin:2} is linear so that 
solve it with respect to $B$,
\begin{EQ}\label{eq:F:nonlin:3:pre}
   B = \delta-A, \qquad \delta=\vartheta_1-\vartheta_0=\phi_1-\phi_0,
\end{EQ}
and the solution of the nonlinear system \eqref{eq:F:nonlin:2:b} is reduced to the
solution of the nonlinear system of two equations in two unknowns, namely $L$ and $A$:
\begin{EQ}\label{eq:F:nonlin:6}
   \bm{G}(L,A)
   =
   \pmatrix{
     \Delta x -L\cdot X(2A,\delta-A,\vartheta_0) \\
     \Delta y -L\cdot Y(2A,\delta-A,\vartheta_0)\\
   },\qquad
\end{EQ}
followed by the computation of $B$ by \eqref{eq:F:nonlin:3:pre}.
We can perform one further simplification using polar coordinates
to represent $(\Delta x,\Delta y)$, namely
\begin{EQ}\label{eq:xhhyhh:polar}
   \cases{
   \Delta x = r\cos\varphi,&\\
   \Delta y = r\sin\varphi.&\\
   }
\end{EQ}
From~\eqref{eq:xhhyhh:polar} and $L>0$ we define two nonlinear
functions $f(L,A)$ and $g(A)$, where $g(A)$ does not depend on $L$, 
as follows:
\begin{EQ}
   f(L,A)=\bm{G}(L,A)\cdot\pmatrix{\cos\varphi\\\sin\varphi},\qquad
   g(A)=\dfrac{1}{L}\bm{G}(L,A)\cdot\pmatrix{\sin\varphi\\-\cos\varphi}.
\end{EQ}
Taking advantage of the trigonometric identities 
\begin{EQ}[rcl]\label{eq:id:sincos}
  \sin(\alpha-\beta)&=&\sin\alpha\cos\beta-\cos\alpha\sin\beta, \\
  \cos(\alpha-\beta)&=&\cos\alpha\cos\beta+\sin\alpha\sin\beta,
\end{EQ}
the functions $g(A)$ and $f(L,A)$ are simplified in:
\begin{EQ}[rcl]\label{eq:g}
  g(A) &=& Y(2A,\delta-A,\phi_0),\\
  f(L,A)
  &=&
  \sqrt{\Delta x^2+\Delta y^2}-L\,X(2A,\delta-A,\phi_0),
\end{EQ}
where  $\phi_0=\vartheta_0-\varphi$.\\
Supposing to find $A$ such that  $g(A)=0$, then from $f(L,A)=0$ we compute $L$ and $B$ using equations \eqref{eq:g} and \eqref{eq:F:nonlin:3:pre} respectively. This yields
\begin{EQ}
  L = \dfrac{\sqrt{\Delta x^2+\Delta y^2}}{X(2A,\delta-A,\phi_0)},
  \qquad
  B = \delta-A.
\end{EQ}
Thus, the solutions of the nonlinear system~\eqref{eq:F:nonlin:2:b} 
are known if the solutions of the single nonlinear function $g(A)$ of equation~\eqref{eq:g}
are determined.
The solution of Problem~\ref{prob:1} is recapitulated in the following steps:
\begin{enumerate}
  \item[1.] Solve $g(A)=0$ where $g(A)\DEF Y(2A,\delta-A,\phi_0)$;
  \item[2.] Compute $L=\sqrt{\Delta x^2+\Delta y^2}/h(A)$ where $h(A)\DEF X(2A,\delta-A,\phi_0)$;
  \item[3.] Compute $\kappa=(\delta-A)/L$ and $\kappa'=2A/L^2$.
\end{enumerate}
This algorithm prompts the following issues.
\begin{itemize}
  \item[$\bullet$] How to compute the roots of $g(A)$ and to select the one
                    which appropriately solves Problem~\ref{prob:1}.
  \item[$\bullet$] Check that the length $L$ is well defined and positive.
\end{itemize}
These issues are discussed in Section~\ref{sec:existence}.
Unlike the work of \citeNP{Meek:2009}, this method
does not require to split the problem in mutually exclusive 
cases, i.e., straight lines and circles are treated naturally.

\section{Theoretical development}\label{sec:existence}
In this section the existence and selection of the appropriate solution
are discussed in detail.
The computation of $L$ requires only
to verify that for $A^\star$ such that $g(A^\star)=0$ then $h(A^\star)= X(2A^\star,\delta-A^\star,\phi_0)\neq 0$.
This does not ensure that the computed $L$ is positive; but positivity is 
obtained by an appropriate choice of $A^\star$.
\subsection{Symmetries of the roots of $g(A)$}
%
The general analysis of the zeros of $g(A)$ requires the angles $\phi_0$ and $\phi_1$ to be
in the range $(-\pi,\pi)$. It is possible to restrict the domain of search stating
the following auxiliary problems:
\paragraph*{The reversed problem}
  The clothoid joining $(x_1,y_1)$ to $(x_0,y_0)$
  with angles $\vartheta_0^R=-\vartheta_1$ and $\vartheta_1^R=-\vartheta_0$ is a curve with 
  support a clothoid that solves Problem~\ref{prob:1}  but running in
  the opposite direction (with the same length $L$).
  Let $\delta^R = \vartheta_1^R-\vartheta_0^R=-\vartheta_0+\vartheta_1 = \delta$, 
  it follows that 
  $g^R(A)\DEF Y(2A,\delta-A,-\phi_1)$ is the function whose zeros give the solution
  of the reversed interpolation problem.

\paragraph*{The mirrored problem}
  The curve obtained connecting $(x_0,y_0)$ to $(x_1,y_1)$
  with angle $\vartheta_0^M=\varphi-\phi_0$ and $\vartheta_1^M=\varphi-\phi_1$ is a curve with 
  support a curve solving the same problem but mirrored along
  the line connecting the points $(x_0,y_0)$ and $(x_1,y_1)$ (with the same length $L$).
  Let $\delta^M = \vartheta_1^M-\vartheta_0^M=-\phi_1+\phi_0 = -\delta$, it follows that 
  $g^M(A)\DEF Y(2A,-\delta-A,-\phi_0)$ is the function whose zeros are the solution
  of the mirrored interpolation problem. \\ \\
Lemma \eqref{lem:reduce} shows that it is possible to reduce the search of the roots in the domain
$\abs{\phi_0}<\phi_1\leq\pi$. The special cases $\phi_0\pm\phi_1=0$ are 
considered separately.
\begin{lemma}\label{lem:reduce}
  Let $g(A)\DEF Y(2A,\delta-A,\phi_0)$ and $h(A)\DEF X(2A,\delta-A,\phi_0)$
  with
  \begin{EQ}[rclcl]
     g^R(A) &\DEF& Y(2A,\delta-A,-\phi_1),\qquad
     g^M(A) &\DEF& Y(2A,-\delta-A,-\phi_0), \\
     h^R(A) &\DEF& X(2A,\delta-A,-\phi_1),\qquad
     h^M(A) &\DEF& X(2A,-\delta-A,-\phi_0),
  \end{EQ}
  then 
  \begin{EQ}
   g(A) = -g^R(-A),\qquad
   g(A) = -g^M(-A),\qquad
   h(A) = h^R(-A) = h^M(-A).
  \end{EQ}
  Thus, $g(A)$ has the same roots of $g^R(A)$, $g^M(A)$
  with opposite sign.
\end{lemma}
\begin{proof}
  \begin{EQ}[rcl]
     g(A) = Y(2A,\delta-A,\phi_0) 
     &=&
     \int_0^1 \sin( A\tau (\tau-1) + \delta \tau +\phi_0) \dtau, \\
     &=&
     -\int_0^1 \sin( -A\tau (\tau-1) - \delta \tau - \phi_0) \dtau, \\
     &=&
     -\int_0^1\sin( A(1-z)z - \delta (1-z) - \phi_0) \dz, \quad [z=1-\tau]\\
     &=&
     -\int_0^1 \sin( -Az(z-1) + \delta z - \phi_1) \dz,\\
     &=& -Y(-2A,\delta+A,-\phi_1), 
  \end{EQ}
  we have $g(A) = -g^R(-A)$.\\
  Using again $g(A)$ with the following manipulation:
  \begin{EQ}[rcl]
     g(-A) 
     =
     Y(-2A,\delta+A,\phi_0)
     &=&
     \int_0^1 \sin( -A\tau^2 + (\delta+A) \tau +\phi_0) \dtau, \\
     &=&
     -\int_0^1 \sin( A\tau^2 + (-\delta-A) \tau - \phi_0) \dtau, \\
     &=& -Y(2A,-\delta-A,-\phi_0)= -g^M(A),
  \end{EQ}
  we have $g(A) = -g^M(-A)$.
  The thesis for $h(A)$ is obtained in the same way.
\end{proof}
\begin{figure}[!cb]
  \begin{center}
    \includegraphics[scale=0.65]{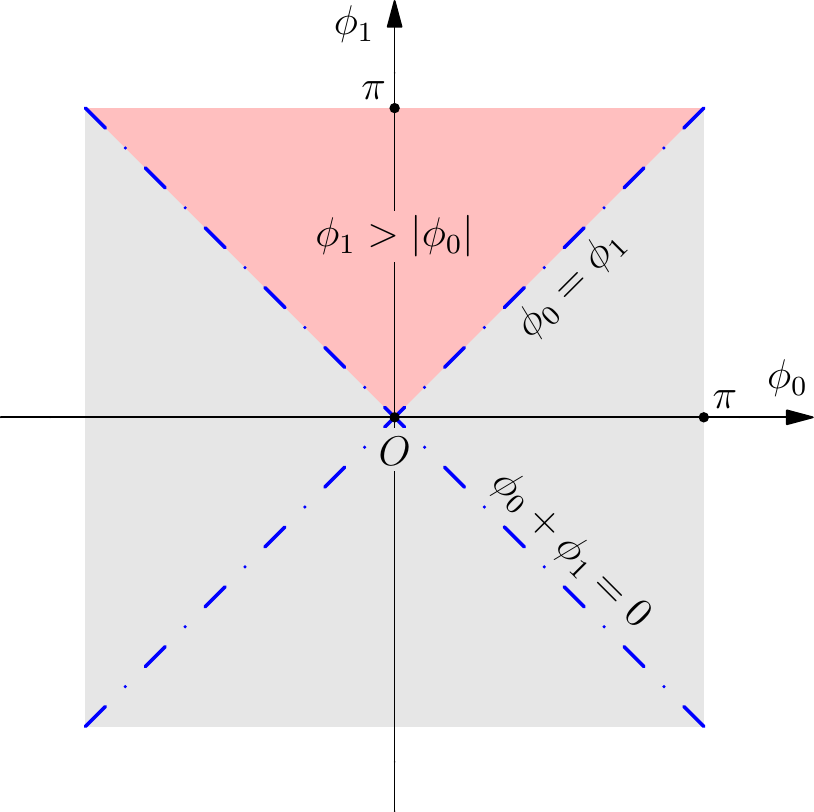}\qquad
    \includegraphics[scale=0.65]{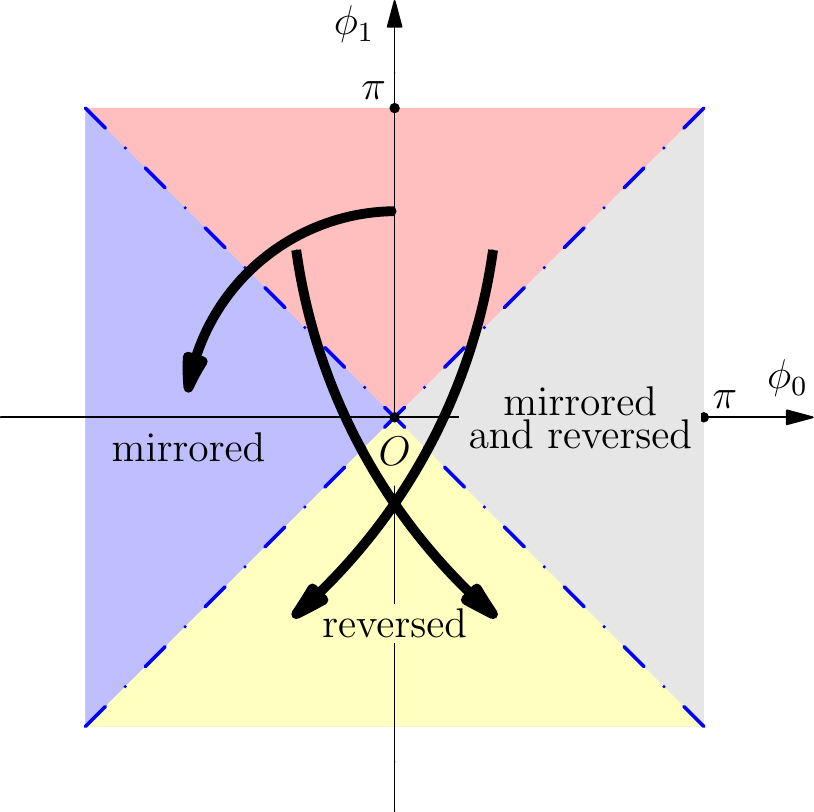}
  \end{center}
  \caption{Left: the domain $\phi_1>\abs{\phi_0}$ with the special cases $\phi_1=\phi_0$
  and $\phi_1+\phi_0=0$. Right: the domain mirrored and reversed.}
  \label{fig:sym}
\end{figure}
Figure~\ref{fig:sym} shows the domain $\abs{\phi_0}<\phi_1\leq\pi$ with
the mirrored and reversed problem.
Reflecting and mirroring allows to assume the constraints for the angles described in Assumption \eqref{ass:ordering}.
\begin{assumption}[Angle domains]\label{ass:ordering}
  The angles $\phi_0$ and $\phi_1$ satisfy the restriction: $\abs{\phi_0}\leq\phi_1\leq\pi$
  with ambiguous cases $\abs{\phi_0}=\phi_1=\pi$ excluded (see Figure~\ref{fig:sym}).
\end{assumption}
This ordering ensures that the curvature of the fitting curve is increasing, i.e. $\kappa'>0$.
Notice that if $A$ is the solution of nonlinear system~\eqref{eq:g} then 
$\kappa'=2A/L^2$, i.e. the sign of $A$ is the sign of $\kappa'$ and thus $A$ must be positive.
Finally, ${\delta=\phi_1-\phi_0>0}$.
This assumption is not a limitation because any interpolation problem can be reformulated
as a problem satisfying Assumption~\ref{ass:ordering}
while the analysis for the special case $\phi_0+\phi_1=0$ and $\phi_0-\phi_1=0$ are performed apart.

\begin{lemma}\label{lem:4.9}
  The (continuous) functions $g(A)\DEF Y(2A,\delta-A,\phi_0)$ 
  and $h(A)\DEF X(2A,\delta-A,\phi_0)$ 
  for $A>0$, when $\phi_0$ and $\phi_1$ satisfy assumption~\ref{ass:ordering},
  can be written as
  \begin{EQ}\label{eq:gA:cases}
    g(A)
    =\dfrac{\sqrt{2\pi}}{\sqrt{A}}
    \cases{ p\left(\frac{(\delta-A)^2}{4A}\right) & $0<A\leq\delta$, \\
             q\left(\frac{(\delta-A)^2}{4A}\right) & $A\geq\delta$;}
    \;
    h(A)
    =\dfrac{\sqrt{2\pi}}{\sqrt{A}}
     \cases{ \overline p\left(\frac{(\delta-A)^2}{4A}\right) & $0<A\leq\delta$, \\
             \overline q\left(\frac{(\delta-A)^2}{4A}\right) & $A\geq\delta$,}
  \end{EQ}
  where
  \begin{EQ}[rclcl]\label{eq:pq}
    p(\theta)&=&\int_{\theta}^{\theta+\delta}\dfrac{\sin(u+\phi_0-\theta)}{\sqrt{u}}\du,
    \quad
    q(\theta)
    &=&p(\theta)+2\int_{0}^{\theta}\dfrac{\sin(u+\phi_0-\theta)}{\sqrt{u}}\du,
    \\
    \overline p(\theta)&=&\int_{\theta}^{\theta+\delta}\dfrac{\cos(u+\phi_0-\theta)}{\sqrt{u}}\du,
    \quad
    \overline q(\theta)
    &=&\overline p(\theta)+2\int_{0}^{\theta}\dfrac{\cos(u+\phi_0-\theta)}{\sqrt{u}}\du.\quad
  \end{EQ}
\end{lemma}
\begin{proof}
With standard trigonometric passages and $A>0$ we deduce the following expression for $g(A)$ and $h(A)$:
\begin{EQ}[rcl]
  \sqrt{A}\,
  g(A)
  &=&
  \sqrt{2\pi}
  \Big[
    \left(\Cf(\omega_+)-\Cf(\omega_-)\right)\sin\eta +
    \left(\Sf(\omega_+)- \Sf(\omega_-)\right)\cos\eta
  \Big],
  \\
  \sqrt{A}\,
  h(A)
  &=&
  \sqrt{2\pi}
  \Big[
    \left(\Cf(\omega_+)-\Cf(\omega_-)\right)\cos\eta -
    \left(\Sf(\omega_+)- \Sf(\omega_-)\right)\sin\eta
  \Big],
\end{EQ}
where
\begin{EQ}[c]\label{eq:th:2}
  \omega_- = \frac{\delta-A}{\sqrt{2\pi A}}, \qquad
  \omega_+ = \frac{\delta+A}{\sqrt{2\pi A}}, \qquad
  \eta = \phi_0-\dfrac{(\delta-A)^2}{4A}.
\end{EQ}
Combining equivalence~\eqref{eq:equiv}
and the parity properties of $\sin x$ and $\cos x$, $g(A)$ and $h(A)$ take the form:
\begin{EQ}[rcl]\label{eq:g:and:h}
  \sqrt{A}\,
  g(A)
  &=&
  \Delta\widehat\Cf\sin\left(\phi_0-\theta\right)+\Delta\widehat\Sf\cos\left(\phi_0-\theta\right),
  \\
  \sqrt{A}\,
  h(A)
  &=&
  \Delta\widehat\Cf\cos\left(\phi_0-\theta\right)-\Delta\widehat\Sf\sin\left(\phi_0-\theta\right),
\end{EQ}
where
\begin{EQ}
  \Delta\widehat\Cf=\widehat\Cf\left(\theta+\delta\right)-\sigma_-\widehat\Cf\left(\theta\right),
  \qquad
  \Delta\widehat\Sf=\widehat\Sf\left(\theta+\delta\right)-\sigma_-\widehat\Sf\left(\theta\right),
\end{EQ}
and $\widehat\Cf$, $\widehat\Sf$ are defined in \eqref{eq:SC:alternative}; moreover
\begin{EQ}\label{eq:theta:A}
         \theta = \frac{(\delta-A)^2}{4A},
   \qquad \theta+\delta = \frac{(\delta+A)^2}{4A},
   \qquad \sigma_-    = \SIGN(\delta-A).\qquad
\end{EQ}
By using identities \eqref{eq:id:sincos} equation~\eqref{eq:g:and:h} becomes:
\begin{EQ}[rcl]
  \hat g(\theta) =\dfrac{\sqrt{A}}{\sqrt{2\pi}} g(A)
  &=&
  \int_0^{\theta+\delta}
  \dfrac{\sin(u+\phi_0-\theta)}{\sqrt{u}}\du
  -
  \sigma_-\int_0^{\theta}
  \dfrac{\sin(u+\phi_0-\theta)}{\sqrt{u}}\du,
  \\
  \hat h(\theta) =\dfrac{\sqrt{A}}{\sqrt{2\pi}} h(A)
  &=&
  \int_0^{\theta+\delta}
  \dfrac{\cos(u+\phi_0-\theta)}{\sqrt{u}}\du
  -
  \sigma_-\int_0^{\theta}
  \dfrac{\cos(u+\phi_0-\theta)}{\sqrt{u}}\du.
\end{EQ}
It is recalled that $A$ must be positive, so that when $A$ to $0 < A <\delta$ then $\sigma_-=1$,
otherwise, when $A>\delta$ then $\sigma_-=-1$. In case $A=\delta$ then $\theta=0$ and the second
integral is $0$ and thus $g(\delta)=p(0)=q(0)$ and $h(\delta)=\overline p(0)=\overline q(0)$.
\end{proof}

\subsection{Localization of the roots of $g(A)$}

The problem $g(A)=0$ has in general infinite solutions.
The next Theorems show the existence of a 
\emph{unique} solution in a prescribed range, 
they are in part new and in part taken from \cite{Walton:2009} 
here reported without proofs and notation slightly changed to better match the notation.
By appropriate transformations we use these
Theorems to select the suitable solution
and find the interval where the solution is unique.
The Theorems characterize the zeros of the functions~\eqref{eq:pq}
finding intervals where the solution exists and is unique.

\begin{theorem}[Meek-Walton th.2]\label{th:MW:2}
Let $0<-\phi_0<\phi_1<\pi$. If $p(0)>0$ then $p(\theta)=0$ has no root for $\theta\geq 0$.
If $p(0)\leq 0$ then $p(\theta)=0$ has exactly one root for $\theta\geq 0$.
Moreover, the root occurs in the interval $[0,\theta^\star]$ where
\begin{EQ}\label{eq:cond:theta}
  \theta^\star=\dfrac{\lambda^{2}}{1-\lambda^{2}}(\phi_{1}-\phi_{0})>0
  \qquad 
  0<\lambda=\dfrac{1-\cos\phi_{0}}{1-\cos\phi_{1}}<1.
\end{EQ}
\end{theorem}
%
\begin{theorem}[Meek-Walton th.3]\label{th:MW:3}
Let 
$-\pi < -\phi_1 < \phi_0 < 0$ and $q(0)>0$ 
then $q(\theta)=0$ has 
exactly one root in the interval $[0,\pi/2+\phi_0]$.
If $q(0)<0$ then $q(\theta)=0$ has 
no roots in the interval $[0,\pi/2+\phi_0]$.
\end{theorem}
%
\begin{theorem}[Meek-walton th.4]\label{th:MW:4}
  Let $\phi_0\in[0,\pi)$ and $\phi_1\in(0,\pi]$, then $q(\theta)=0$  
  has exactly one root in $[0,\pi/2+\phi_0]$, moreover, 
  the root occurs in $[\phi_0,\pi/2+\phi_0]$. 
\end{theorem}
The following additional Lemmata are necessary to complete
the list of properties of $p(\theta)$ and $q(\theta)$:
\begin{samepage}
\begin{lemma}\label{lem:4.8}
  Let $p(\theta)$ and $q(\theta)$ as defined in equation~\eqref{eq:pq}, then
  \begin{itemize}
    \item[$(a)$] if $0\leq\phi_0\leq\phi_1\leq\pi$ then
    \begin{itemize}
      \item[$\circ$] if $\phi_1>\phi_0$ then $p(\theta)>0$ for all $\theta\geq0$ 
                       otherwise $p(\theta)=0$ for all $\theta\geq0$;
      \item[$\circ$] $q(\theta)=0$ for $\theta\in[\phi_0,\pi/2+\phi_0]$
                       and the root is unique in the interval $[0,\pi/2+\phi_0]$;
    \end{itemize}
    \item[$(b)$] if $-\pi\leq-\phi_1<\phi_0<0$
    \begin{itemize}
      \item[$\circ$] if $p(0)=q(0)\leq 0$ then
      \begin{itemize}
        \item[$\bullet$] $p(\theta)=0$ has a unique root $\theta$ that satisfies
        $0\leq\theta\leq\theta_0$ with $\theta_0$ defined in~\eqref{eq:cond:theta}.
        \item[$\bullet$] $q(\theta)=0$ has no roots in the interval $[0,\pi/2+\phi_0]$;
      \end{itemize}
      \item[$\circ$] if $p(0)=q(0)>0$ then
      \begin{itemize}
        \item[$\bullet$] $p(\theta)>0$ for all $\theta\geq 0$;
        \item[$\bullet$] $q(\theta)=0$ has a unique root in the interval $[0,\pi/2+\phi_0]$
      \end{itemize}
    \end{itemize}
    \item[$(c)$] if $\phi_0\leq-\pi/2$ then $p(0)=q(0)<0$.
  \end{itemize}
\end{lemma}
\end{samepage}
\begin{proof}
  A straightforward application of Theorems~\ref{th:MW:2}, \ref{th:MW:3} and \ref{th:MW:4}.
  For point $(c)$, from~\eqref{eq:g:and:h}:
  \begin{EQ}
    p(0)=q(0)=
    \sqrt{\delta}\,
    g(\delta)
    =
    \Delta\widehat\Cf\sin\phi_0+\Delta\widehat\Sf\cos\phi_0,
  \end{EQ}
  in addiction, since $-\pi \leq \phi_0 \leq -\pi/2$, both $\sin\phi_0\leq 0$ and $\cos\phi_0\leq 0$
  resulting in $p(0)=q(0)<0$.
\end{proof}

The combination of Lemma~\ref{lem:reduce} together with reversed and mirrored problems, proves that the interpolation problem satisfies
Assumption~\ref{ass:ordering}.

With this Assumption and Lemma~\ref{lem:4.9} prove a Theorem that states the existence and uniqueness in a specified range. We show the special case of $\phi_{0}+\phi_{1}=0$ in a separated Lemma, the case of $\phi_{0}-\phi_{1}=0$ follows from the application of Theorem \ref{th:MW:4} for positive angles, because Assumption~\ref{ass:ordering} forces $\phi_{1}\geq 0$ and excludes the case of equal negative angles.

\begin{lemma}\label{lem:special1}
  Let $\phi_0+\phi_1=0$ and $\phi_0\in(-\pi,\pi)$, 
  then $g(A)=0$ has the unique solution $A=0$
  in the interval $(-2\pi,2\pi)$.
\end{lemma}
\begin{proof}
  For $\phi_0+\phi_1=0$ we have $\delta=-2\phi_0$ and
  \begin{EQ}[rcl]
     g(A) &=& Y(2A,-2\phi_0-A,\phi_0), \\
     &=&
     \int_0^1 \sin( A\tau (\tau-1) +\phi_0(1-2\tau)) \dtau, \\
     &=&
     \int_{-1}^1 \sin( A(z^2-1)/4 - z\phi_0) \frac{\dz}{2}, \qquad[\tau=(z+1)/2]\\
     &=&
     \int_{-1}^1 \sin( A(z^2-1)/4 )\cos(z\phi_0) \frac{\dz}{2}
     -
     \int_{-1}^1 \cos( A(z^2-1)/4 )\sin(z\phi_0) \frac{\dz}{2}.
  \end{EQ}
  Using properties of odd and even functions the rightmost integral of the previous line vanishes yielding
  \begin{EQ}
     g(A) = \int_0^1 \sin( A(z^2-1)/4 )\cos(z\phi_0) \dz.
  \end{EQ}
  From this last equality, if $A=0$ then $g(A)=0$.
  If $0<\abs{A}<4\pi$, the sign of the quantity $\sin( A(z^2-1)/4 )$ is constant;
  if $\abs{\phi_0}<\pi/2$, then $\cos(z\phi_0)>0$ and thus
  $g(A)$ has no roots. For the remaining values of $\phi_0$, i.e. $\pi/2\leq\abs{\phi_0}<\pi$, we have:
  \begin{EQ}
     \int_0^{\pi/(2\abs{\phi_0})}\cos(z\phi_0) \dz
     = \dfrac{1}{\abs{\phi_0}},
     \qquad
     \int_{\pi/(2\abs{\phi_0})}^1\abs{\cos(z\phi_0)} \dz
     =\dfrac{1-\sin\abs{\phi_0}}{\abs{\phi_0}}<\dfrac{1}{\abs{\phi_0}}.
  \end{EQ}
  If in addition, $0<\abs{A}<2\pi$ then $\abs{\sin( A(z^2-1)/4 )}$ is positive and 
  monotone decreasing so that:
  \begin{EQ}[rcl]
     \bigg|\int_0^{\pi/(2\abs{\phi_0})} \sin\Big(\dfrac{A}{4}(z^2-1)\Big)\cos(z\phi_0) \dz\bigg|
     &\geq& 
     \dfrac{C}{\abs{\phi_0}},
     \\
     \bigg|\int_{\pi/(2\abs{\phi_0})}^1 \sin\Big(\dfrac{A}{4}(z^2-1)\Big)\cos(z\phi_0) \dz\bigg|
     &<&
     \dfrac{C}{\abs{\phi_0}},
  \end{EQ}
  where
  \begin{EQ}
     C=\Big|\sin\Big(\dfrac{A}{16\abs{\phi_0}^2}\left(\pi^2-4\abs{\phi_0}^2\Big)\right)\Big|>0,
  \end{EQ}
  and thus $g(A)\neq 0$ for $0<\abs{A}<2\pi$ and $\abs{\phi_0}<\pi$.
\end{proof}

We state now the main Theorem of this study.
\begin{theorem}[Existence and uniqueness of solution for system~\eqref{eq:g}]\label{teo:ex:uniq}
  The function \allowbreak$g(A)$, when angles $\phi_0$ and $\phi_1$ satisfy assumption~\ref{ass:ordering},
  admits a unique solution for $A\in(0,A_{\max}]$, where
  \begin{EQ}\label{eq:A:interval}
     A_{\max} = \delta + 2\theta_{\max}\left(1+\sqrt{1+\delta/\theta_{\max}}\right),
     \qquad
     \theta_{\max} = \max\Big\{0,\pi/2+\phi_0\Big\}.
  \end{EQ}
  Moreover $h(A)>0$ where $g(A)=0$.
\end{theorem}
\begin{proof}
The special cases $\phi_{0}+\phi_{1}=0$ and $\phi_{0}=\phi_{1}$ were previously considered  
and in Lemma~\ref{lem:special1}. The other cases are discussed next.
From Lemma~\ref{lem:4.8} it follows that the two equations
\begin{EQ}[rcll]
  p(\theta) &=& 0, \qquad &\theta \geq 0, \\
  q(\theta) &=& 0, \qquad &\theta \in[0,\theta_{\max}],
\end{EQ}
cannot be satisfied by the same $\theta$ in the specified range, so that they are mutually exclusive although one of the two is satisfied. Thus $g(A)=0$ has a unique solution. To find the equivalent range of $A$ we select the correct solution of $(\delta-A)^2=4A\theta_{\max}$. The two roots are:
  \begin{EQ}[rcl]\label{eq:Asol}
   A_1 &=& 2\theta_{\max}+\delta-2\sqrt{\theta_{\max}^2+\theta_{\max}\delta}
     = \dfrac{\delta^2}{2\theta_{\max}+\delta+2\sqrt{\theta_{\max}^2+\theta_{\max}\delta}}\leq \delta
     \\
   A_2 &=& 2\theta_{\max}+\delta+2\sqrt{\theta_{\max}^2+\theta_{\max}\delta}\geq \delta,
  \end{EQ}
  and thus $A_2$ is used to compute $A_{\max}$. We are now interested to check if $h(A)>0$ when $g(A)=0$. According to Lemma \ref{lem:4.9},
  \begin{EQ}
   h(A)
    =\cases{ \overline p\left(\frac{(\delta-A)^2}{4A}\right) = \displaystyle\int_{\phi_0}^{\phi_1}\dfrac{\cos z}{\sqrt{z+\theta-\phi_0}}\dz, & $0<A\leq\delta$, \\
             \overline q\left(\frac{(\delta-A)^2}{4A}\right), & $A\geq\delta$.}
  \end{EQ}  
Suppose we set the root  $\theta$ of $g$ (by doing so, we set also the value of $A$), we have to show that if $A\leq \delta$ then $\overline p(\theta)>0$ and if $A\geq \theta$ then $\overline q(\theta)>0$.\\
 We begin analysing $A\leq \delta$ for $\abs{\phi_0}<\phi_1\leq \frac{\pi}{2}$. In this case we have that the cosine in the numerator is always positive, and so is the square root at the denominator, thus the integral $\overline p(\theta)$ is strictly positive. When $-\pi<\phi_0<-\frac{\pi}{2}$ then $\phi_1\geq \frac{\pi}{2}$, we can write, for all $w\in\mathbb{R}$, 
\begin{EQ}\label{int:sincos}
  \int_{\phi_0}^{\phi_1}\dfrac{\cos z}{\sqrt{z+\theta-\phi_0}}\dz
  = \int_{\phi_0}^{\phi_1}\dfrac{\cos z-w\sin z}{\sqrt{z+\theta-\phi_0}}\dz.
\end{EQ}
In particular we can choose $w=\frac{\cos \phi_0}{\sin \phi_0}>0 $ positive (which, incidentally is always positive because $-\pi < \phi_{0}< -\frac{\pi}{2}$) so that the integrand function vanishes for the three values $z=\phi_0$, $\phi_0+\pi$, $\phi_0+2\pi$. Moreover $\cos z-w\sin z$ is strictly positive
for $z\in(\phi_0,\phi_0+\pi)$, see Figure~\ref{fig:h}.
\begin{figure}[!tcb]
  \begin{center}
    \includegraphics[scale=0.7]{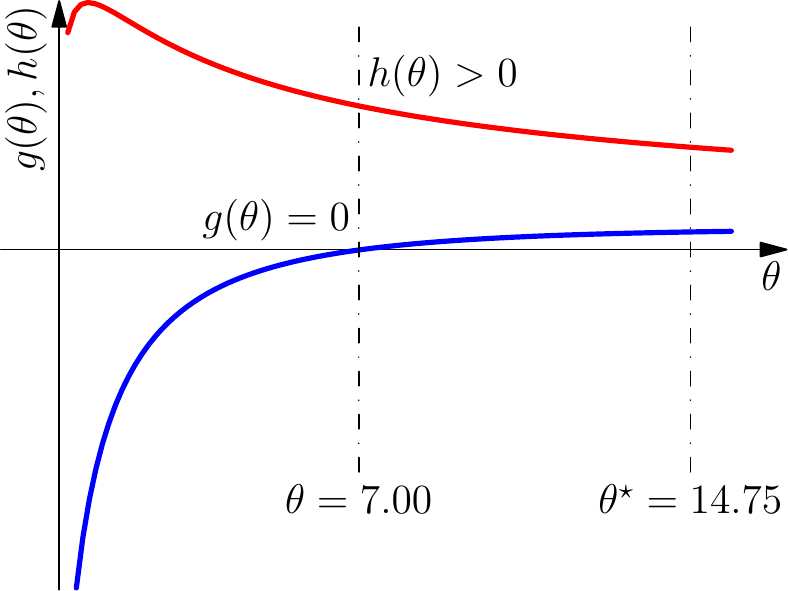} \qquad
    \includegraphics[scale=0.7]{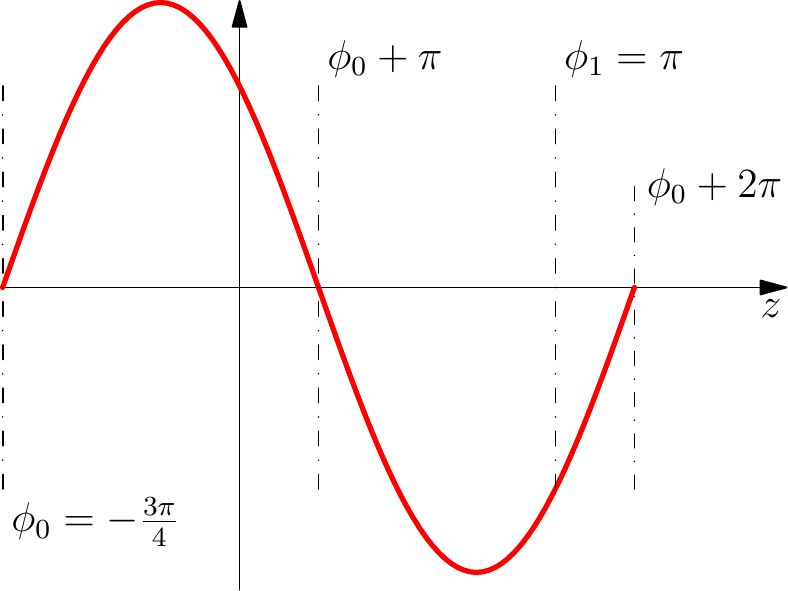}
  \end{center}
  \caption{Left: functions $g$ and $h$, when $g$ vanishes, $h$ is strictly
           positive. Right: the plot of $\cos z-w\sin z$. In both figures
           $\phi_0=-(3/4)\pi$ and $\phi_1 = \pi$.
           }
  \label{fig:h}
\end{figure}
Thus we can bound integral \eqref{int:sincos} as
\begin{EQ}[rcl]
 \int_{\phi_0}^{\phi_1}\dfrac{\cos z-w\sin z}{\sqrt{z+\theta-\phi_0}}\dz &>&
 \int_{\phi_0}^{\phi_0+2\pi}\dfrac{\cos z-w\sin z}{\sqrt{z+\theta-\phi_0}}\dz,\\ 
& \geq &  \dfrac{\int_{\phi_0}^{\phi_0+2\pi}\cos z-w\sin z\dz}{\sqrt{(\phi_0+\pi)+\theta-\phi_0}}=0.
\end{EQ}
The last subcase of $A\leq \delta$ occurs when both $\phi_0,\phi_1>\frac{\pi}{2}$, but it is not necessary because Lemma \ref{lem:4.8} at point (a) states that $p(\theta)>0$ does not vanish. We discuss now the second case, when $A\geq \delta$, it is recalled that from \eqref{eq:g:and:h}
\begin{EQ}
h(A)=\Delta\widehat\Cf\cos\left(\phi_0-\theta\right)-\Delta\widehat\Sf\sin\left(\phi_0-\theta\right),
\end{EQ}
with $\Delta\widehat\Cf,\Delta\widehat\Sf>0$ and for $\theta\in [0,\frac{\pi}{2}+\phi_0]$ we can write $-\frac{\pi}{2}\leq \phi_0-\theta\leq 0$, thus the cosine is positive and the sine is negative, hence the whole quantity is strictly positive.
\end{proof}

\begin{corollary}\label{lem:LAB}
  All the solutions of the nonlinear system \eqref{eq:F:nonlin:6}
  are given by
  \begin{EQ}\label{eq:LAB}
    L = \dfrac{\sqrt{\Delta x^2+\Delta y^2}}{ X(2A,\delta-A,\phi_0)},\qquad
    \kappa = \dfrac{\delta-A}{L}, \qquad
    \kappa' = \dfrac{2\,A}{L^2},\qquad
  \end{EQ}  
  where $A$ is any root of $g(A)\DEF Y(2A,\delta-A,\phi_0)$ provided that the corresponding 
  $h(A)\DEF X(2A,\delta-A,\phi_0)>0$.
\end{corollary}  
\begin{corollary}\label{lem:LAB:1}
  If the angles $\phi_0$ and $\phi_1$ are 
  in the range $[-\pi,\pi]$, with the exclusion of the points $\phi_0=-\phi_1=\pm\pi$,
  the solution exists and is unique 
  for $-A_{\max} \leq A \leq A_{\max}$ where 
  \begin{EQ}[rcl]
     A_{\max} &=& \abs{\phi_1-\phi_0} + 2\theta_{\max}\left(1+\sqrt{1+\abs{\phi_1-\phi_0}/\theta_{\max}}\right),
     \\
     \theta_{\max} &=& \max\Big\{0,\pi/2+\SIGN(\phi_1)\phi_0\Big\}.
  \end{EQ}
\end{corollary}
Notice that the solution of the nonlinear
system~\eqref{eq:F:nonlin:6} is reduced to the solution of $g(A)=0$.
Hence the interpolation problem is reduced 
to a single \emph{nonlinear} equation that can be solved 
numerically with Newton--Raphson Method; also, Corollary~\ref{lem:LAB:1}
ensures that $L>0$ for $A$ in the specified range.
It can also be proved that in general for all $A$ such that $g(A)=0$ 
then $h(A)\neq 0$.
In the next section, a technique to select a meaningful initial
point is described.

\section{Computing the initial guess for iterative solution}
\label{sec:5}
\begin{figure}[!bcb]
  \begin{center}
    \includegraphics[scale=0.42]{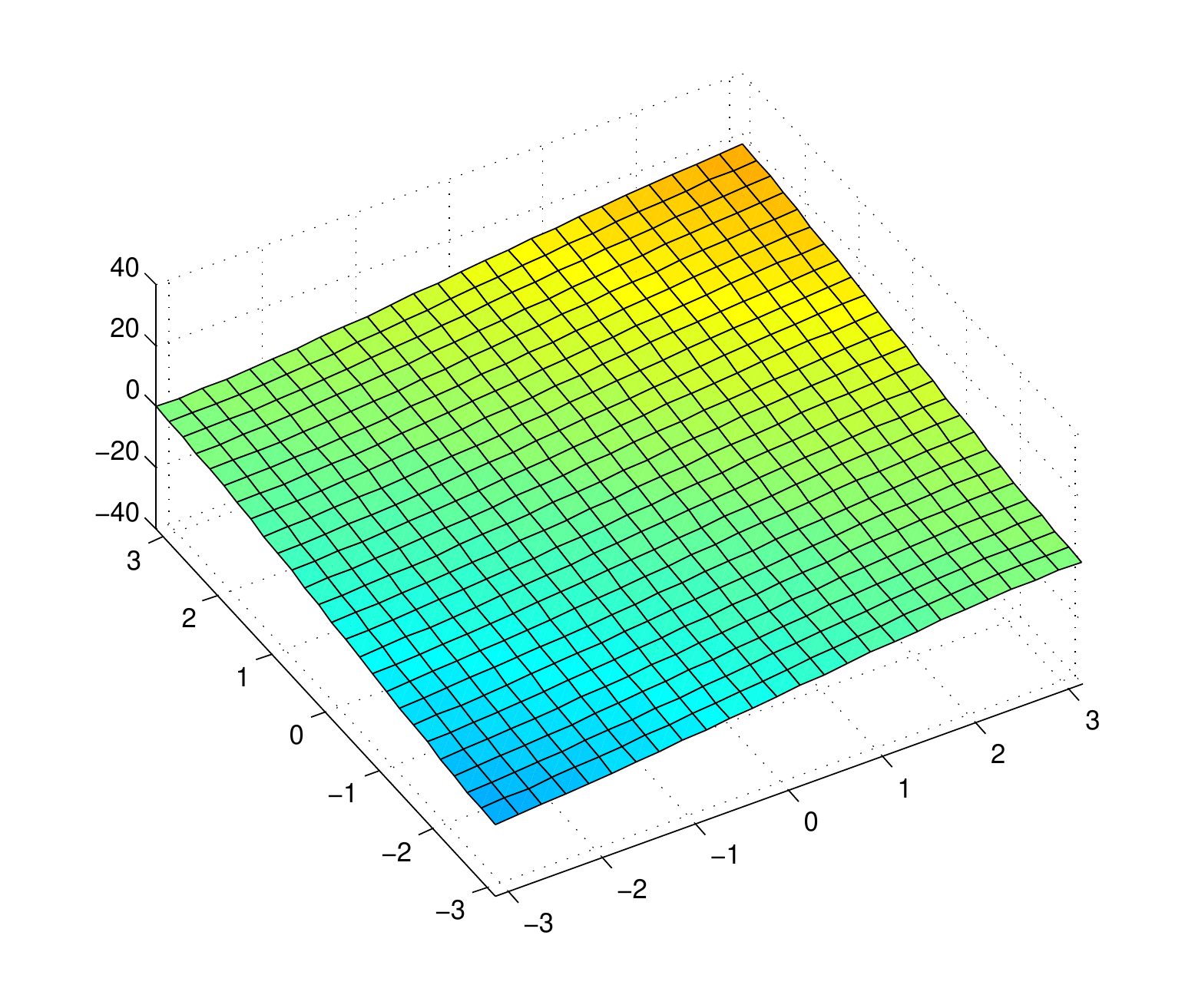}
    \includegraphics[scale=0.42]{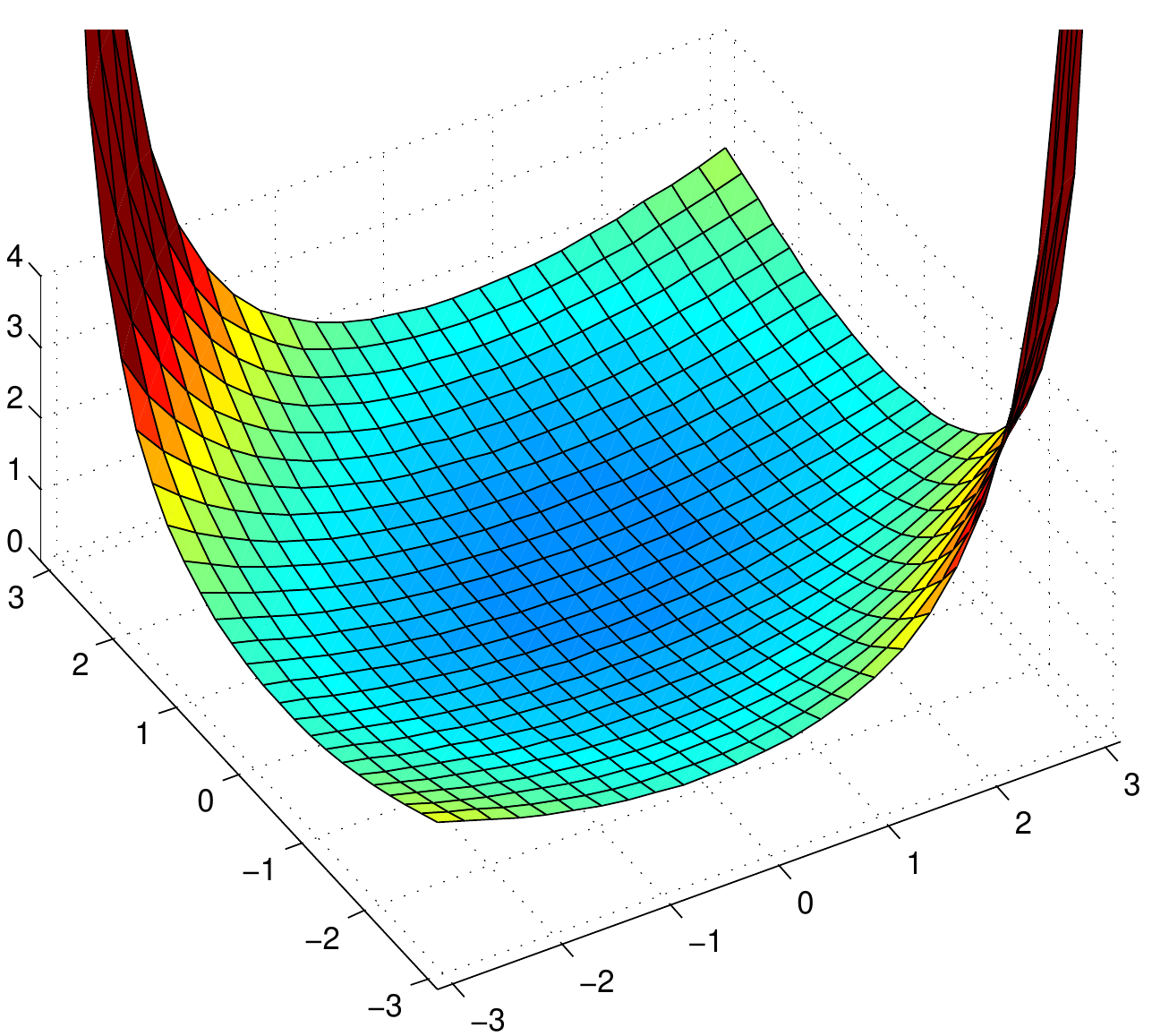}
  \end{center}
  \caption{Left: the function $\mathcal{A}(\phi_0,\phi_1)$.
           Notice that $\mathcal{A}(\phi_0,\phi_1)$ is approximately a plane.
           Right: values of the length $L$ of the computed clothoid as a function of $\phi_0$ and $\phi_1$.
           Notice that when angles satisfy $\phi_0=\pi,\,\phi_1=-\pi$ or $\phi_0=-\pi,\,\phi_1=\pi$
           the length goes to infinity. 
           The angles range in $[-\pi,\pi]$.           
           }
  \label{fig:AL}
\end{figure}
The zeros of function $g(A)$ are used to solve the interpolation problem 
and are approximated by the Newton-Raphson scheme. This algorithm needs ``a guess point'' 
to converge to the appropriate solution.
Notice that there is an infinite number of solutions of Problem~\ref{prob:1}
and we need criteria for the selection of a solution.
By Theorem~\ref{teo:ex:uniq} we reduce the problem to standard angles
and search the unique solution in the appropriate range.

Denote with $\mathcal{A}(\phi_0,\phi_1)$ the selected zero of $g(A)\DEF Y(2A,\delta-A,\phi_0)$
as a function of $\phi_0$ and $\phi_1$.
Figure~\ref{fig:AL} shows that $\mathcal{A}(\phi_0,\phi_1)$ is approximated by a plane.
A simple approximation of $\mathcal{A}(\phi_0,\phi_1)$ is obtained by  
$\sin x \approx x$ in $Y(2A,\delta-A,\phi_0)$ and thus,
\begin{EQ}
   g(A) = Y(2A,\delta-A,\phi_0)
        \approx \int_0^1 A\tau^2 + (\delta-A)\tau + \phi_0 \dtau
        =\dfrac{\phi_0+\phi_1}{2} -\dfrac{A}{6},
\end{EQ}
and solving for $A$
\begin{EQ}\label{eq:A:approx:a}
  \mathcal{A}(\phi_0,\phi_1)\approx 3(\phi_0+\phi_1),
\end{EQ}
This approximation is a fine initial point for Newton-Raphson, however
better approximation for $\mathcal{A}(\phi_0,\phi_1)$ are obtained
by least squares method.
Invoking reflection and mirroring properties, the functional form of the approximation
is simplified and results in the two following expressions
for $\mathcal{A}(\phi_0,\phi_1)$:
\begin{EQA}[rcl]\yesnumber
  \mathcal{A}(\phi_0,\phi_1) &\approx& (\phi_0+\phi_1)
  \Big(c_1+c_2\overline{\phi}_0\overline{\phi}_1+c_3(\overline{\phi}_0^2+\overline{\phi}_1^2)\Big),
  \label(~a){eq:A:approx:b}
  \\
  \mathcal{A}(\phi_0,\phi_1) &\approx&
  (\phi_0+\phi_1)\Big(d_1+\overline{\phi}_0\overline{\phi}_1(d_2+d_3\overline{\phi}_0\overline{\phi}_1)+ \\ &&
  \qquad\qquad\quad(\overline{\phi}_0^2+\overline{\phi}_1^2)(d_4+d_5\overline{\phi}_0\overline{\phi}_1)+d_6(\overline{\phi}_0^4+\overline{\phi}_1^4)\Big),
  \label(~b){eq:A:approx:c}
\end{EQA}
where $\overline{\phi}_0=\phi_0/\pi$, $\overline{\phi}_1=\phi_1/\pi$ and 
\begin{center}
  \begin{tabular}{c|cccccc}
        & 1 & 2 & 3 & 4 & 5 & 6 \\
    \hline
     $c$
     &  3.070645
     &  0.947923
     & -0.673029
     \\
     $d$
     &  2.989696
     &  0.71622
     & -0.458969
     & -0.502821
     &  0.26106
     & -0.045854 
  \end{tabular}
\end{center}

Using~\eqref{eq:A:approx:a}, \eqref{eq:A:approx:b} or \eqref{eq:A:approx:c} 
as the starting point for Newton-Raphson, the solution for Problem~\ref{prob:1} is found in very few 
iterations.
Computing the solution with Newton-Raphson starting with the proposed guesses
in a $1024\times 1024$ grid for $\phi_0$ and $\phi_1$ ranging in $[-0.9999\pi,0.9999\pi]$ 
with a tolerance of $10^{-10}$, results in the following distribution of iterations:
\begin{center}
\begin{tabular}{c|ccccc|l}
  \#iter  & 1 & 2 & 3 & 4 & 5 & \\
  \hline
  \#instances & 1025    & 6882    & 238424   & 662268   & 142026  & Guess~\eqref{eq:A:approx:a} \\
          & $0.1\%$ & $0.7\%$ & $22.7\%$ & $63.0\%$ & $13.5\%$ & \\
  \hline
  \#instances & 1025    & 10710    & 702534   & 336356   &         & Guess~\eqref{eq:A:approx:b} \\
          & $0.1\%$ & $1.0\%$ & $66.9\%$ & $32.0\%$ &         & \\
  \hline
  \#instances & 1025    & 34124 & 1015074   & 402 & & Guess~\eqref{eq:A:approx:c} \\
          & $0.1\%$ & $3.2\%$ & $96.6\%$ & $0.1\%$ &  & 
\end{tabular}
\end{center}

\begin{table}[!bt]
\caption{The fitting algorithm}
\label{tab:algo:complete}
\begin{function}[H]
  \caption{normalizeAngle($\phi$)}
  \label{alg:normalizeAngle}
  \def\assign{\leftarrow}
  \SetAlgoLined 
  \small
  \lWhile{$\phi>+\pi$}{$\phi\assign\phi-2\pi$}\;
  \lWhile{$\phi<-\pi$}{$\phi\assign\phi+2\pi$}\;
  \Return{$\phi$}\;
\end{function}
\vspace{-0.9em}
\begin{function}[H]
  \caption{buildClothoid($x_0$, $y_0$, $\vartheta_0$, $x_1$, $y_1$, $\vartheta_1$, $\epsilon$)}
  \label{alg:buildClothoid}
  \SetKwFunction{findA}{findA}
  \SetKwFunction{Aguess}{Aguess}
  \SetKwFunction{normalizeAngle}{normalizeAngle}
  \def\assign{\leftarrow}
  \small
  \SetAlgoLined 
  $\Delta x \assign x_1-x_0$;\quad
  $\Delta y \assign y_1-y_0$\;
  Compute $r$ and $\varphi$ from $r\cos\varphi=\Delta x$ and $r\sin\varphi=\Delta y$\;
  $\ASSIGNm{\phi_0}\normalizeAngle(\vartheta_0-\varphi)$;\quad
  $\ASSIGNm{\phi_1}\normalizeAngle(\vartheta_1-\varphi)$\;
  Define $g$ as $g(A):=Y(2A,(\phi_1-\phi_0)-A,\phi_0)$\;
  Set $A\assign 3(\phi_1+\phi_0)$;\quad\tcp{In alternative use \eqref{eq:A:approx:b} or \eqref{eq:A:approx:c}}
  \lWhile{$\abs{g(A)}>\epsilon$}{
    $A\assign A - g(A)/g'(A)$
  }\;
  $\ASSIGNs{L}r/X(2A,\delta-A,\phi_0)$;\quad
  $\ASSIGNs{\kappa}(\delta-A)/L$;\quad
  $\ASSIGNs{\kappa'}(2\,A)/L^2$\;
  \Return{$\kappa$, $\kappa'$, $L$}
\end{function}
\end{table}
The complete algorithm for the clothoid computation is written in the function
\reffun{alg:buildClothoid} of Table~\ref{tab:algo:complete}.
This function solves equation \eqref{eq:g}  and 
builds the coefficients of the interpolating clothoid.\\
The algorithm is extremely compact and was successfully tested in any possible situation.
The accurate computation of the clothoid needs an equally accurate computation of $g(A)$
and $g'(A)$ and thus the accurate computation of Fresnel 
related functions $X(a,b,c)$ and $Y(a,b,c)$ with associated
derivatives.
These functions are a combination of Fresnel and Fresnel 
momenta integrals which are precise for large $a$ and small momenta.
For the computation, only the first two momenta are necessary
so that the inaccuracy for large momenta does not pose any problem.
A different problem is the computation of these integrals for
small values of $\abs{a}$. In this case, laborious expansion are used
to compute the integrals with high accuracy.
The expansions are discussed in section \ref{sec:6}.

\section{Accurate computation of Fresnel momenta}\label{sec:6}

Here it is assumed that the standard Fresnel integrals can be computed 
with high accuracy. For this task one can use algorithms described in 
\citeNP{Snyder:1993}, \citeNP{Smith:2011} and \citeNP{Thompson:1997} 
or using continued fraction expansion as in \citeNP{Backeljauw:2009}
or simply use the available software \cite{NR:2002}.
It is possible to reduce the integrals \eqref{clot} 
to a linear combination of this standard Fresnel integrals \eqref{fresnel}.
However simpler expressions are obtained using also the momenta of the Fresnel integrals:
\begin{EQ}\label{fresnel:moments}
  \Cf_k(t)=\int_0^t\!\!\!\tau^k\cos \left(\frac{\pi}{2}\tau^2\right)\dtau,\qquad
  \Sf_k(t)=\int_0^t\!\!\!\tau^k\sin \left(\frac{\pi}{2}\tau^2\right)\dtau.\qquad
\end{EQ} 
Notice that $\Cf(t)\DEF \Cf_0(t)$ and $\Sf(t)\DEF\Sf_0(t)$.
Closed forms via the exponential integral or the Gamma function are also possible,
however we prefer to express them as a recurrence.  
Integrating by parts, the following recurrence is obtained:
\begin{EQ}[rcl]\label{eq:fresnel:recurrence}
  \Cf_{k+1}(t) & = &
  \dfrac{1}{\pi}\left(
    t^{k}\sin\left(\frac{\pi}{2}t^2\right)-k\,\Sf_{k-1}(t)
    \right),\\
  \Sf_{k+1}(t) & = &
  \dfrac{1}{\pi}\left(
    k\,\Cf_{k-1}(t)-t^{k}\cos\left(\frac{\pi}{2}t^2\right)\right).
\end{EQ} 
Recurrence is started by computing standard Fresnel integrals~\eqref{fresnel}
and (changing  $z=\tau^2$) the following values are obtained:
\begin{EQ}
  \Cf_1(t) = \frac{1}{\pi}\sin \left(\frac{\pi}{2}t^2\right),
  \qquad
  \Sf_1(t) = \frac{1}{\pi}\left(1-\cos\left(\frac{\pi}{2}t^2\right)\right).
\end{EQ} 
Also $\Cf_2(t)$ and $\Sf_2(t)$ are readily obtained:
\begin{EQ}
  \Cf_2(t) = \dfrac{1}{\pi}\left(t\sin\left(\frac{\pi}{2}t^2\right)-\Sf_0(t)\right),\qquad
  \Sf_2(t) = \dfrac{1}{\pi}\left(\Cf_0(t)-t\cos\left(\frac{\pi}{2}t^2\right)\right).
\end{EQ} 
Notice that from recurrence~\eqref{eq:fresnel:recurrence} it 
follows that $\Cf_k(t)$ and $\Sf_k(t)$ with $k$ odd 
do not contain Fresnel integrals \eqref{fresnel}
and are combination of elementary functions.

The computation of clothoids relies most on the evaluation of integrals of kind 
\eqref{eq:XY0} with their derivatives. The reduction is possible via a change of variable and integration by parts. It is sufficient to consider two integrals, that cover all possible cases:
\begin{EQ}[rcl]\label{eq:fresnel:general}
  X_k(a,b,c)&=&\int_0^1 \tau^k\cos\left(\frac{a}{2}\tau^2+b\tau+c\right)\dtau,\\
  Y_k(a,b,c)&=&\int_0^1 \,\tau^k\sin\left(\frac{a}{2}\tau^2+b\tau+c\right)\dtau,
\end{EQ}
so that
\begin{EQ}[rcl]
   g(A)  &=& Y_0(2A,\delta-A,\phi_0),\\
   g'(A) &=& X_1(2A,\delta-A,\phi_0)-X_2(2A,\delta-A,\phi_0),
\end{EQ}
and finally, equation \eqref{clot} can be evaluated as
\begin{EQ}[rcl]
  x(s) &=& x_0 + s\,X_0(\kappa's^2,\kappa s,\vartheta_0), \\
  y(s) &=& y_0 + s\,Y_0(\kappa's^2,\kappa s,\vartheta_0).
\end{EQ}
From the trigonometric identities~\eqref{eq:id:sincos},
integrals \eqref{eq:fresnel:general} are rewritten as
\begin{EQ}[rcl]
  X_k(a,b,c)&=&X_k(a,b,0)\cos c-Y_k(a,b,0)\sin c, \\
  Y_k(a,b,c)&=&X_k(a,b,0)\sin c+Y_k(a,b,0)\cos c.
\end{EQ}
Defining $X_k(a,b)\DEF X_k(a,b,0)$ and $Y_k(a,b)\DEF Y_k(a,b,0)$
the computation of \eqref{eq:fresnel:general}
is reduced to the computation of $X_k(a,b)$ and $Y_k(a,b)$.
It is convenient to introduce the following quantities 
\begin{EQ}
  \sigma = \SIGN(a),\qquad
  z =\omega_+-\omega_-=\sigma\frac{\sqrt{\abs{a}}}{\sqrt{\pi}},\qquad
  \omega_-=\frac{b}{\sqrt{\pi\abs{a}}},\qquad
  \eta=-\frac{b^2}{2a},
\end{EQ}
so that it is possible to rewrite the argument of the trigonometric
functions of $X_k(a,b)$ and $Y_k(a,b)$ as
\begin{EQ}[rcl]
   \frac{a}{2}\tau^2+b\tau&=&
   \dfrac{\pi}{2}\sigma
   \left(\tau \dfrac{\sigma\sqrt{\abs{a}}}{\sqrt{\pi}}+\frac{b}{\sqrt{\pi\abs{a}}}\right)^2
   -\frac{b^2}{2a}
   =
   \dfrac{\pi}{2}\sigma\big(\tau z+\omega_-\big)^2+\eta.
\end{EQ}
By using the change of variable $\xi=\tau\,z+\omega_-$ with inverse $\tau= z^{-1}(\xi-\omega_-)$
for $X_k(a,b)$ and the identity~\eqref{eq:id:sincos} we have:
\begin{EQ}[rcl]\label{eq:X:comput}
  X_k(a,b) &=& 
  z^{-1}\int_{\omega_-}^{\omega_+}
  z^{-k}(\xi-\omega_-)^k\cos\left(\sigma\frac{\pi}{2}\xi^2+\eta\right)\dxi,
  \\
  &=& 
  z^{-k-1}
  \int_{\omega_-}^{\omega_+}  
  \sum_{j=0}^k \binom{k}{j}\,\xi^j(-\omega_-)^{k-j}\cos\left(\frac{\pi}{2}\xi^2+\sigma\eta\right)\dxi,
  \\
  &=& z^{-k-1}\sum_{j=0}^k \binom{k}{j}\,(-\omega_-)^{k-j}\left[\cos\eta \Delta\Cf_j-\sigma\sin\eta\Delta\Sf_j\right],
  \\
  &=&
  \dfrac{\cos\eta}{z^{k+1}}
  \bigg[
  \sum_{j=0}^k \binom{k}{j}(-\omega_-)^{k-j}\Delta\Cf_j
  \bigg]
  -
  \sigma\dfrac{\sin\eta}{z^{k+1}}
  \bigg[
  \sum_{j=0}^k \binom{k}{j}(-\omega_-)^{k-j}\Delta\Sf_j
  \bigg]
\end{EQ}
where
\begin{EQ}\label{eq:def:Cj:Sj}
  \Delta\Cf_j= \Cf_j(\omega_+)-\Cf_j(\omega_-), \qquad \Delta\Sf_j= \Sf_j(\omega_+)-\Sf_j(\omega_-),
  \qquad
\end{EQ}
are the evaluation of the momenta of the Fresnel integrals as defined in \eqref{fresnel:moments}.
Analogously for $Y_k(a,b)$ we have:
\begin{EQ}\label{eq:Y:comput}
  Y_k(a,b) = 
  \dfrac{\sin\eta}{z^{k+1}}
  \bigg[\sum_{j=0}^k \binom{k}{j}(-\omega_-)^{k-j}\Delta\Cf_j\bigg]
  +
  \sigma\dfrac{\cos\eta}{z^{k+1}}
  \bigg[\sum_{j=0}^k \binom{k}{j}(-\omega_-)^{k-j}\Delta\Sf_j\bigg].\qquad
\end{EQ}
This computation is inaccurate when $\abs{a}$ is small: in fact $z$ appears in the denominator of several fractions.
For this reason, for small values of $\abs{a}$ we substitute~\eqref{eq:X:comput} and \eqref{eq:Y:comput}
with asymptotic expansions.
Notice that the recurrence \eqref{eq:fresnel:recurrence} is unstable so that it produces inaccurate results for large $k$,
but we need only the first two terms so this 
is not a problem for the computation of $g(A)$ and $g'(A)$.

\subsection{Accurate computation with small parameters}
When the parameter $a$ is small, we use 
identity~\eqref{eq:id:sincos} to 
derive series expansion:
\begin{EQ}[rcl]\label{eq:serie:X:asmall}
   X_k(a,b)
   &=&
   \int_0^1\tau^k\cos\left(\frac{a}{2}\tau^2+b\tau\right)\dtau, \\
   &=&
   \int_0^1\tau^k\left[
   \cos\left(\frac{a}{2}\tau^2\right)\cos(b\tau)
   -
   \sin\left(\frac{a}{2}\tau^2\right)\sin(b\tau)
   \right]\dtau,
   \\
   &=&
   \sum_{n=0}^\infty\frac{(-1)^n}{(2n)!}\left(\frac{a}{2}\right)^{2n}
   X_{4n+k}(0,b)
   -
   \sum_{n=0}^\infty\frac{(-1)^n}{(2n+1)!}\left(\frac{a}{2}\right)^{2n+1}
   Y_{4n+2+k}(0,b),
   \\
   &=&
   \sum_{n=0}^\infty\frac{(-1)^n}{(2n)!}\left(\frac{a}{2}\right)^{2n}
   \left[
   X_{4n+k}(0,b)-\dfrac{a\,Y_{4n+2+k}(0,b)}{2(2n+1)}
   \right],
\end{EQ}
and analogously using again identity~\eqref{eq:id:sincos}
we have the series expansion
\begin{EQA}[rcl]\label{eq:serie:Y:asmall}
   Y_k(a,b) &=&
   \int_0^1\tau^k\sin\left(\frac{a}{2}\tau^2+b\tau\right)\dtau \\
   &=&
   \sum_{n=0}^\infty\frac{(-1)^n}{(2n)!}\left(\frac{a}{2}\right)^{2n}
   \left[
   Y_{4n+k}(0,b)+\dfrac{a\,X_{4n+2+k}(0,b)}{2(2n+1)}
   \right].
\end{EQA}
From the inequalities:
\begin{EQ}
  \abs{X_{k}}\leq \int_{0}^{1}|\tau^{k}|\dtau=\dfrac{1}{k+1},
  \qquad
  \abs{Y_{k}}\leq \int_{0}^{1}|\tau^{k}|\dtau=\dfrac{1}{k+1},
\end{EQ}
we estimate the remainder for the series of $X_k$:
\begin{EQ}[rcl]
   R_{p,k} &=& 
   \abs{
   \sum_{n=p}^\infty\frac{(-1)^n}{(2n)!}\left(\frac{a}{2}\right)^{2n}
   \left[
   X_{4n+k}(0,b)-\dfrac{a\,Y_{4n+2+k}(0,b)}{2(2n+1)}
   \right]}
   \\
   &\leq&
   \sum_{n=p}^\infty\frac{1}{(2n)!}\left(\frac{a}{2}\right)^{2n}
   \left[\dfrac{1}{4n+1}+\dfrac{\abs{a}}{2(2n+1)(4n+3)}
   \right]
   \\
   &\leq&
   \left(\frac{a}{2}\right)^{2p}
   \sum_{n=p}^\infty\frac{1}{(2(n-p))!}\left(\frac{a}{2}\right)^{2(n-p)}
   \\
   &\leq&
   \left(\frac{a}{2}\right)^{2p}
   \sum_{n=0}^\infty\frac{1}{(2n)!}\left(\frac{a}{2}\right)^{2n}
   =
   \left(\frac{a}{2}\right)^{2p}\cosh(a).
\end{EQ}
The same estimate is obtained for the series of $Y_k$
\begin{remark}\label{rem:epsilon}
Both series \eqref{eq:serie:X:asmall} and \eqref{eq:serie:Y:asmall}
converge fast. For example, if $\abs{a}<10^{-4}$ and $p=2$, the error is less than $6.26\cdot 10^{-18}$
while if $p=3$ the error is less than $1.6\cdot 10^{-26}$.
\end{remark}
It is possible to compute 
$X_k(0,b)$ and $Y_k(0,b)$ as follows:
\begin{EQ}[rcll]\label{eq:XY:recu:a:small}
   X_0(0,b) &=& b^{-1}\sin b, &\\
   Y_0(0,b) &=& b^{-1}(1-\cos b),&\\
   X_k(0,b) &=& b^{-1}\big(\sin b-k\,Y_{k-1}(0,b)\big),\qquad & k=1,2,\ldots\\
   Y_k(0,b) &=& b^{-1}\big(k\,X_{k-1}(0,b)-\cos b\big),\qquad & k=1,2,\ldots
\end{EQ}
This recurrence permits the computation of $X_k(0,b)$ and $Y_k(0,b)$ when $b\neq 0$, however, it is readily seen that this recurrence formula is not stable.
As an alternative, an explicit formula based on Lommel function
$s_{\mu,\nu}(z)$ can be used \cite{Shirley:2003}.
The explicit formula is, for $k=1,2,3,\ldots$:
\begin{EQ}[rcl]\label{eq:int:XY0b}
   X_k(0,b) &=&
   \dfrac{
     k\,s_{k+\frac{1}{2},\frac{3}{2}}(b)\sin b +f(b)s_{k+\frac{3}{2},\frac{1}{2}}(b)
   }{(1+k)b^{k+\frac{1}{2}}}
   + \dfrac{\cos b}{1+k},\qquad
   \\
   Y_k(0,b) &=&
   \dfrac{
     k\,s_{k+\frac{3}{2},\frac{3}{2}}(b)\sin b+g(b)s_{k+\frac{1}{2},\frac{1}{2}}(b)
   }{(2+k)b^{k+\frac{1}{2}}}
   + \dfrac{\sin b}{2+k},
\end{EQ}
where $f(b) = b^{-1}\sin b-\cos b$ and $g(b)=f(b)(2+k)$. The Lommel function
has the following expansion (see~\url{http://dlmf.nist.gov/11.9} or reference~\cite{Watson:1944})
\begin{EQ}\label{eq:lommel}
  s_{\mu,\nu}(z) 
  =
  z^{\mu+1}\sum_{n=0}^\infty\dfrac{(-z^2)^n}{\alpha_{n+1}(\mu,\nu)},
  \qquad
  \alpha_n(\mu,\nu)=\prod_{m=1}^n ((\mu+2m-1)^2-\nu^2),\qquad
\end{EQ}
and using this expansion in~\eqref{eq:int:XY0b} results in the following
explicit formula for $k=1,2,3,\ldots$:
\begin{EQ}[rcl]
   X_k(0,b) &=&
   A(b)w_{k+\frac{1}{2},\frac{3}{2}}(b)+
   B(b)w_{k+\frac{3}{2},\frac{1}{2}}(b)+
   \dfrac{\cos b}{1+k},\qquad
   \\
   Y_k(0,b) &=&
   C(b)w_{k+\frac{3}{2},\frac{3}{2}}(b)+
   D(b)w_{k+\frac{1}{2},\frac{1}{2}}(b)+
   \dfrac{\sin b}{2+k},\qquad
\end{EQ}
where
\begin{EQ}
   w_{\mu,\nu}(b) =
   \sum_{n=0}^\infty\dfrac{(-b^2)^n}{\alpha_{n+1}(\mu,\nu)}, \quad
   A(b) = \dfrac{kb\sin b}{1+k},\quad
   B(b) = \dfrac{(\sin b-b\cos b)b}{1+k},\\
   C(b) = -\dfrac{b^2\sin b}{2+k},\quad
   D(b) = \sin b-b\cos b.
\end{EQ}

\section{Numerical tests}
\label{sec:7}

The algorithm was implemented and tested in MATLAB (2009 version).
For the Fresnel integrals computation we use the script of 
\citeN{Telasula:2005}.
The first six tests are taken form \citeN{Walton:2008} and a MATLAB implementation
of the algorithm described in the reference is used for comparison.
\begin{description}
  \item[Test 1] $(x_0,y_0)=(5,4)$, $(x_1,y_1)=(5,6)$, $\vartheta_0=\pi/3$,   $\vartheta_1=7\pi/6$;
  \item[Test 2] $(x_0,y_0)=(3,5)$, $(x_1,y_1)=(6,5)$, $\vartheta_0=2.14676$, $\vartheta_1=2.86234$;
  \item[Test 3] $(x_0,y_0)=(3,6)$, $(x_1,y_1)=(6,6)$, $\vartheta_0=3.05433$, $\vartheta_1=3.14159$;
  \item[Test 4] $(x_0,y_0)=(3,6)$, $(x_1,y_1)=(6,6)$, $\vartheta_0=0.08727$, $\vartheta_1=3.05433$;
  \item[Test 5] $(x_0,y_0)=(5,4)$, $(x_1,y_1)=(4,5)$, $\vartheta_0=0.34907$, $\vartheta_1=4.48550$;
  \item[Test 6] $(x_0,y_0)=(4,4)$, $(x_1,y_1)=(5,5)$, $\vartheta_0=0.52360$, $\vartheta_1=4.66003.$
\end{description}
The accuracy of fit as in \citeN{Walton:2008} 
is determined by comparing the ending point as computed by both
methods, with the given ending point.
For all the tests, both methods have an error which does not exceed $10^{-15}$ or less.
Also iterations are comparable and are reported in the following table.
\begin{center}
\begin{tabular}{l|cccccc}
  \#test & 1 & 2 & 3 & 4 & 5 & 6 \\
  \hline
  \#iter present method & 3 & 3 & 3 & 3 & 3 & 3 \\
  \#iter \citeN{Walton:2008} & 5 & 4 & 4 & 4 & 5 & 4 \\
\end{tabular}
\end{center}
\begin{figure}[!b]
  \begin{center}
    \includegraphics[scale=0.5]{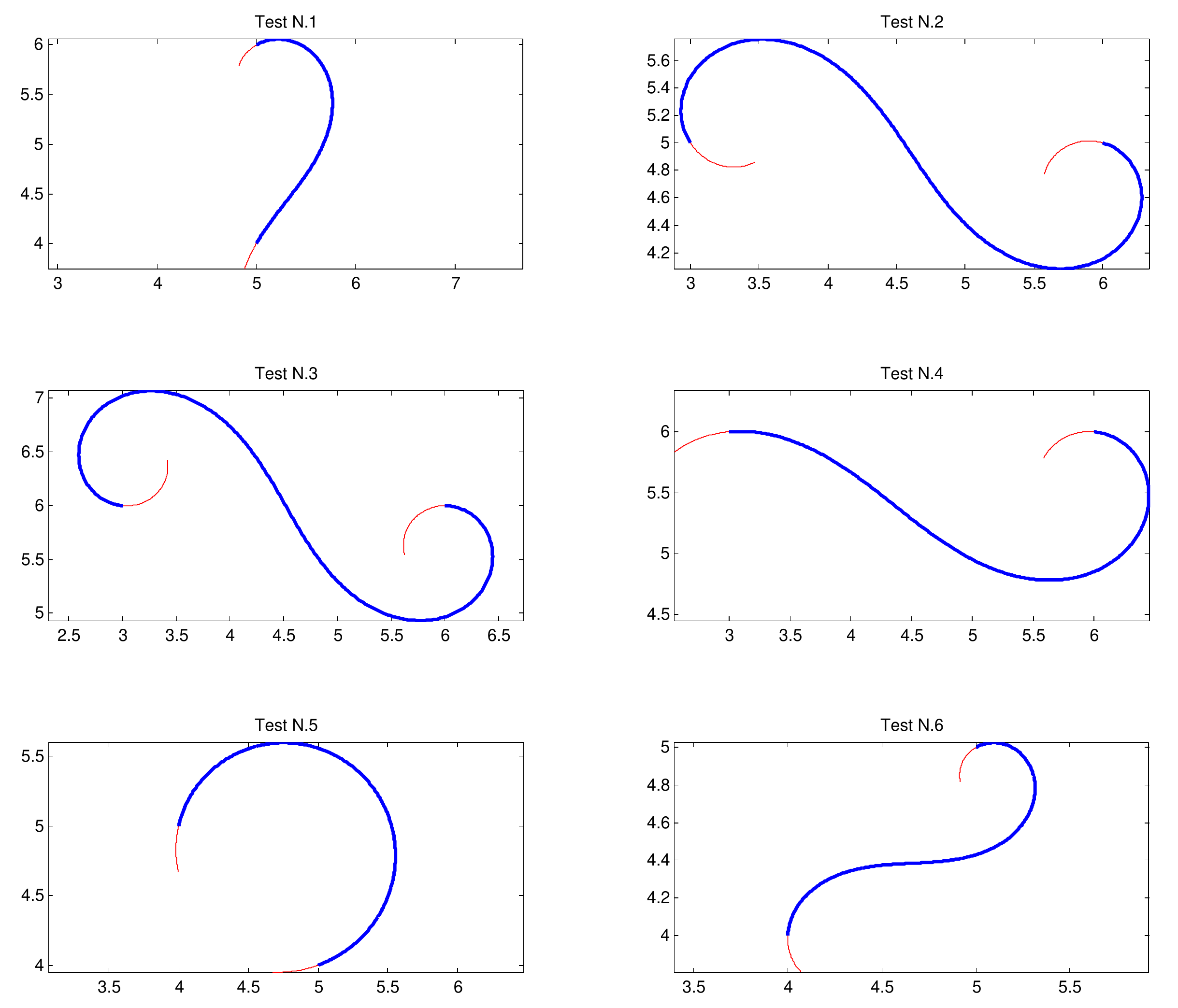}
  \end{center}
  \caption{Results of test N.1 up to test N.6.}
\end{figure}
The difference of the present method compared with the algorithm of \citeN{Walton:2008}
are in the transition zone where the solution is close to be a circle arc or a segment.
In fact in this situation the present method performs better without loosing  accuracy
or increasing in iterations as for the Meek and Walton algorithm.
The following tests highlight the differences:
\begin{description}
  \item[Test 7] $(x_0,y_0)=(0,0)$, $(x_1,y_1)=(100,0)$,
                $\vartheta_0=0.01\,\cdot\,2^{-k}$, $\vartheta_1=-0.02\,\cdot\,2^{-k}$;
                for $k=1,2,\ldots,10$;
  \item[Test 8] $(x_0,y_0)=(0,-100)$, $(x_1,y_1)=(-100,0)$, \\
                $\vartheta_0=0.00011\,\cdot\,2^{-k}$, $\vartheta_1=\frac{3}{2}\pi-0.0001\,\cdot\,2^{-k}$
                for $k=1,2,\ldots,10$.
\end{description}
Table~\ref{tab:7and8} collects the results.
The error is computed as the maximum of the norm of the
differences at the ending point as computed by the algorithm
with the given ending point. The tolerance used for the Newton 
iterative solver for both the algorithms is $\numprint{1E-12}$.
Improving the tolerance does not improve the error for 
both the methods.
Notice that the proposed algorithm computes the solution
with constant accuracy and few iterations while Meek and Walton algorithm
loose precision and uses more iterations.
The $\infty$ symbol for iterations in Table~\ref{tab:7and8}  means that 
Newton method do not reach the required accuracy and the solution
is computed using the last computed values. The iteration limit used
was $100$, increasing this limit to $1000$ did not change the results.
\begin{table}[!htcb]
\caption{Test \#7 and \#8 results}
\label{tab:7and8}
\begin{center}
\begin{tabular}{|l||c|r||c|r||c|r||c|r|}
  \hline
  & \multicolumn{4}{c||}{Test \#7} &\multicolumn{4}{c|}{Test \#8} \\
  & \multicolumn{2}{c}{Proposed method} 
  & \multicolumn{2}{c||}{Meek \& Walton}
  & \multicolumn{2}{c}{Proposed method} 
  & \multicolumn{2}{c|}{Meek \& Walton}
  \\
  \hline
  $k$ 
  & iter & \multicolumn{1}{c||}{Error} & iter & \multicolumn{1}{c||}{Error}
  & iter & \multicolumn{1}{c||}{Error} & iter & \multicolumn{1}{c|}{Error} 
  \\
  \hline
  1  & $2$ & \numprint{2.6e-16}  & $30$ & \numprint{1.83e-6} & $3$ & \numprint{3.18e-14} & $\infty$ & \numprint{3.76e-9} \\
  2  & $2$ & \numprint{1.42e-14} & $29$ & \numprint{1.85e-6} & $3$ & \numprint{2.01e-14} & $\infty$ & \numprint{1.45e-8} \\
  3  & $3$ & 0                   & $28$ & \numprint{1.38e-6} & $2$ & \numprint{2.01e-14} & $\infty$ & \numprint{7.47e-8} \\
  4  & $2$ & \numprint{4.33e-17} & $27$ & \numprint{9.83e-7} & $2$ & \numprint{2.84e-14} & $\infty$ & \numprint{3.47e-8} \\
  5  & $2$ & \numprint{5.42e-18} & $26$ & \numprint{6.96e-7} & $2$ & \numprint{0}        & $\infty$ & \numprint{1.07e-9} \\
  6  & $2$ & 0                   & $25$ & \numprint{4.92e-7} & $2$ & \numprint{1.42e-14} & $\infty$ & \numprint{5.53e-9} \\
  7  & $2$ & \numprint{1.35e-18} & $24$ & \numprint{3.48e-7} & $2$ & \numprint{5.12e-14} & $\infty$ & \numprint{2.43e-7} \\
  8  & $2$ & 0                   & $23$ & \numprint{2.46e-7} & $2$ & \numprint{0}        & $\infty$ & \numprint{3.09e-6} \\
  9  & $2$ & 0                   & $22$ & \numprint{1.74e-7} & $2$ & \numprint{0}        & $\infty$ & \numprint{3.25e-6} \\
  10 & $2$ & 0                   & $21$ & \numprint{1.23e-7} & $2$ & \numprint{5.12e-14} & $\infty$ & \numprint{4.84e-7} \\
  \hline
\end{tabular}
\end{center}
\end{table}
In Table~\ref{tab:7and8}, the algorithm of \citeN{Walton:2008} has a large number of iteration respect to the 
proposed method. To understand this behavior, it is recalled that in solving a general nonlinear function $f(\theta)=0$ using Newton-Raphson method,
the error $\epsilon_k = \theta_k-\alpha$ near the root $\theta=\alpha$ satisfies:
\begin{EQ}
   \epsilon_{k+1} \approx C\epsilon_{k}^2, \qquad C = -\dfrac{f''(\alpha)}{2f'(\alpha)}.
\end{EQ}
showing the quadratic behavior of Newton-Raphson method; 
large values of $C$ reflects a slow convergence.
If $f(\theta)$ is the function used by \citeN{Walton:2009} for computing 
the clothoid curve and $g(A)$ is the function of the proposed method then
\begin{EQ}
  f(\theta(A))\dfrac{\sqrt{A}}{\sqrt{2\pi}}=g(A),\qquad \theta(A)=\dfrac{(\delta-A)^2}{4A},
\end{EQ}
and the error $e_k = A_k-A^{\star}$ near the root $\theta(A^{\star})=\alpha$ satisfy:
\begin{EQ}
   e_{k+1} \approx C^{\star}e_{k}^2, \qquad C^\star = -\dfrac{g''(A^{\star})}{2g'(A^{\star})}.
\end{EQ}
If $g(A^\star)=0$ and $A^\star>0$ but small (i.e. $A^\star\approx 0$)
then $\theta(A)\approx\delta^2/(4A)$ near $A^\star$,
moreover
\begin{EQ}
   C^\star\approx \dfrac{2(A^\star)^2}{\delta^2}\dfrac{f''(\alpha)}{f'(\alpha)},
   \qquad\Rightarrow\qquad
   C\approx-\dfrac{\delta^2}{4(A^\star)^2}C^\star,
\end{EQ}
thus if the ratio $\delta/A^\star$ is large then $\abs{C}\gg\abs{C^\star}$
and thus \citeN{Walton:2008} algorithm is slower than the proposed one. This
is verified in Table~\ref{tab:7and8}. Notice that Test \#7 has constant $C$
lower than Test \#7 and this is reflected in the iteration counts.
When the ratio $\delta/A^\star$  is small then $\abs{C}\ll\abs{C^\star}$ 
and the algorithm of Meek \& Walton should be faster. An estimation based on Taylor
expansion shows that when $\delta/A^\star$ is small then $C^\star$ is small too.
Thus the proposed algorithm do not suffer of slow convergence as verified experimentally.

\section{Conclusions}

An effective solution to the problem of Hermite $G^1$ interpolation 
with a clothoid curve was presented.
We solve the $G^1$ Hermite interpolation problem in a \emph{new and complete} way.

In our approach we do not need the decomposition in mutually exclusive states,
that, numerically, introduces instabilities and inaccuracies as showed in section~\ref{sec:7}.
The solution of the interpolation problem 
is uniformly accurate even when close to a straight line or an arc of circle,
as pointed out in section~\ref{sec:7}, this was not the case of algorithms found in literature.
In fact, even in domains where other algorithms solve the problem, 
we perform better in terms of accuracy and number of iterations.
The interpolation problem was reduced to one single function in one variable
making the present algorithm fast and robust.
Existence and uniqueness of the problem was discussed and proved in section~\ref{sec:existence}. 
A guess functions which allows to find that zero with 
very few  iterations in all possible configurations was provided.
Asymptotic expansions near critical values for Fresnel related integrals 
are provided to keep uniform accuracy. 
Implementation details of the proposed algorithm are given
in the appendix using pseudocode and can be easily 
translated in any programming language.

\appendix
\section{Algorithms for the computation of Fresnel related integrals}
We present here the algorithmic version of the analytical expression we derived in 
Section~\ref{sec:6} and~\ref{sec:7}.
These algorithms are necessary for the computation of 
the main function \reffun{alg:buildClothoid} of Section~\ref{sec:5} which takes the input data ($x_0$, $y_0$, $\vartheta_0$, $x_1$, $y_1$, $\vartheta_1$) and returns the parameters ($\kappa$, $\kappa'$, $L$) that solve the problem as expressed in equation \eqref{clot}.
Function \reffun{alg:evalXY}
computes the generalized Fresnel integrals \eqref{eq:fresnel:general}.
It distinguishes the cases of $a$ larger or smaller than a threshold $\varepsilon$.
The value of  $\varepsilon$ is discussed in Section~\ref{sec:7},
see for example Remark \ref{rem:epsilon}. 
Formulas~\eqref{eq:X:comput}-\eqref{eq:Y:comput},
used to compute $X_k(a,b)$ and $Y_k(a,b)$ at arbitrary precision when 
$\abs{a}\geq\varepsilon$, are implemented in function \reffun{alg:evalXYaLarge}.
Formulas~\eqref{eq:serie:X:asmall}-\eqref{eq:serie:Y:asmall},
used to compute $X_k(a,b)$ and $Y_k(a,b)$ at arbitrary precision when 
$\abs{a}<\varepsilon$, are implemented in function~\reffun{alg:evalXYaSmall}.
This function requires computation of \eqref{eq:int:XY0b}
implemented in function \reffun{alg:evalXYaZero}
which needs (reduced) Lommel function~\eqref{eq:lommel} implemented in 
function \reffun{alg:S}.

\begin{function}[H]
  \caption{evalXY($a$, $b$, $c$, $k$)}
  \label{alg:evalXY}
  \def\assign{\leftarrow}
  \small
  \SetAlgoLined
  \lIf{$\abs{a}<\varepsilon$}{
    $\hat X,\hat Y\assign\;$\ref{alg:evalXYaSmall}($a$,$b$,$k$,$5$)
  }
  \lElse{
    \qquad\qquad $\;\;\,\hat X,\hat Y\assign\;$\ref{alg:evalXYaLarge}($a$,$b$,$k$)
  }
  \lFor{$j=0,1,\ldots,k-1$}{
    $\ASSIGNs{X_{j}}\hat X_j\cos c-\hat Y_j\sin c$;\quad
    $\ASSIGNs{Y_{j}}\hat X_j\sin c+\hat Y_j\cos c$
  }
  \Return{$X$, $Y$}
\end{function}
\begin{function}[H]
  \caption{evalFresnelMomenta($t$, $k$)}
  \label{alg:evalFresnelMomenta}
  \def\assign{\leftarrow}
  \small
  \SetAlgoLined
  $C_0\assign \Cf(t)$;\quad
  $S_0\assign \Sf(t)$;\quad
  $z\assign \pi (t^2/2)$\quad
  $c\assign \cos z$\quad
  $s\assign \sin z$\;
  \lIf{$k>1$}{
    $C_1\assign s/\pi$;\quad
    $S_1\assign (1-c)/\pi$
  }
  \lIf{$k>2$}{
    $C_2\assign (t\,s-S_0)/\pi$;\quad
    $S_2\assign (C_0-t\,c)/\pi$
  }
  \Return{$C$, $S$}
\end{function}

\begin{function}[H]
  \caption{evalXYaLarge($a$, $b$, $k$)}
  \label{alg:evalXYaLarge}
  \def\assign{\leftarrow}
  \small
  \SetAlgoLined
  $s\assign\dfrac{a}{\abs{a}}$;
  $\;z\assign\sqrt{\abs{a}/\pi}$;
  $\;\ell\assign\dfrac{s\,b}{z\,\pi}$;
  $\;\gamma\assign-\dfrac{sb^2}{2\abs{a}}$;
  $\;s_\gamma\assign\sin\gamma$;
  $\;c_\gamma\assign\cos\gamma$\;
  $C^+,S^+\assign$\ref{alg:evalFresnelMomenta}($\ell+z,k$);\quad
  $\;C^-,S^-\assign$\ref{alg:evalFresnelMomenta}($z,k$)\;
  $\Delta\Cf\assign C^+-C^-$;
  $\;\Delta\Sf\assign S^+-S^-$\;
  $X_0\assign z^{-1}\left(c_\gamma\,\Delta\Cf_0-s\,s_\gamma\,\Delta\Sf_0\right)$;
  $\;Y_0\assign z^{-1}\left(s_\gamma\,\Delta\Cf_0+s\,c_\gamma\,\Delta\Sf_0\right)$\;
  \If{$k>1$}{
    $d_c\assign \Delta\Cf_1-\ell\Delta\Cf_0$;\quad
    $d_s\assign \Delta\Sf_1-\ell\Delta\Sf_0$\;
    $X_1\assign z^{-2}\left(c_\gamma\,d_c-s\,s_\gamma\,d_s\right)$;\quad
    $Y_1\assign z^{-2}\left(s_\gamma\,d_c+s\,c_\gamma\,d_s\right)$\;
  }
  \If{$k>1$}{
    $d_c\assign \Delta\Cf_2+\ell(\ell\Delta\Cf_0-2\Delta\Cf_1)$;\quad
    $d_s\assign \Delta\Sf_2+\ell(\ell\Delta\Sf_0-2\Delta\Sf_1)$\;
    $X_2\assign z^{-3}\left(c_\gamma\,d_c-s\,s_\gamma\,d_s\right)$;\quad
    $Y_2\assign z^{-3}\left(s_\gamma\,d_c+s\,c_\gamma\,d_s\right)$\;
  }
  \Return{$X$, $Y$}
\end{function}

\begin{function}[H]
  \caption{evalXYaZero($b$, $k$)}
  \label{alg:evalXYaZero}
  \def\assign{\leftarrow}
  \small
  \SetAlgoLined
  \uIf{$\abs{b}<\varepsilon$}{
    $X_0\assign1-(b^2/6)\left(1-b^2/20\right)$;\quad
    $Y_0\assign (b^2/2)\left(1-(b^2/6)\left(1-b^2/30\right)\right)$\;
  }
  \Else{
    $\ASSIGNs{X_0}(\sin b)/b$;\quad
    $\ASSIGNs{Y_0}(1-\cos b)/b$\;
  } 
  $\ASSIGNs{A}b\sin b$;\quad
  $\ASSIGNs{D}\sin b-b\cos b$;\quad
  $\ASSIGNs{B}bD$;\quad
  $\ASSIGNs{C}-b^2\sin b$\;
  \For{$k=0,1,\ldots,k-1$}{
    $\ASSIGNl{X_{k+1}}
    \dfrac{1}{1+k}\left(
    kA\,\reffun{alg:S}\big(k+\frac12,\frac32,b\big) +
                         B\,\reffun{alg:S}\big(k+\frac32,\frac12,b\big) + \cos b\right)$\;
    $\ASSIGNl{Y_{k+1}}
    \dfrac{1}{2+k}\left(
    C\,\reffun{alg:S}\big(k+\frac32,\frac32,b\big) + \sin b\right)
     + D\,\reffun{alg:S}\Big(k+\frac12,\frac12,b\Big)$\;
  }
  \Return{$X$, $Y$}
\end{function}

\begin{function}[H]
  \caption{evalXYaSmall($a$, $b$, $k$, $p$)}
  \label{alg:evalXYaSmall}
  \SetKwFunction{findA}{findA}
  \SetKwFunction{Aguess}{Aguess}
  \SetKwFunction{normalizeAngle}{normalizeAngle}
  \def\assign{\leftarrow}
  \small
  \SetAlgoLined
  $\hat X,\hat Y\assign$~\ref{alg:evalXYaZero}$(b, k+4p+2)$;\quad
  $t\assign 1$\;
  \lFor{$j=0,1,\ldots,k-1$}{
    $\ASSIGNs{X_j}X^0_j-\dfrac{a}{2}Y^0_{j+2}$;\quad
    $\ASSIGNs{Y_j}Y^0_j+\dfrac{a}{2}X^0_{j+2}$
  }
  \For{$n=1,2,\ldots,p$}{
    $t\assign(-t\,a^2)/(16n(2n-1))$;\quad $s\assign a/(4n+2)$\;
    \For{$j=0,1,\ldots,k-1$}{
      $X_j\assign X_j + t(\hat X_{4n+j}-s\,\hat Y_{4n+j+2})$\;
      $Y_j\assign Y_j + t(\hat Y_{4n+j}+s\,\hat X_{4n+j+2})$;
    }
  }
  \Return{$X$, $Y$}
\end{function}

\begin{function}[H]
  \caption{rLommel($\mu$, $\nu$, $b$)}
  \label{alg:S}
  \SetKw{AND}{and}
  \SetKw{BREAK}{break}
  \def\assign{\leftarrow}
  \small
  \SetAlgoLined

  $\ASSIGNs{t}(\mu+\nu+1)^{-1}(\mu-\nu+1)^{-1}$;\quad
  $\ASSIGNs{r}t$;\quad$n\assign 1$;\quad$\varepsilon\assign 10^{-50}$\;
  \lWhile{$\abs{t}>\varepsilon\abs{r}$}{
    $t\assign t \dfrac{(-b)}{2n+\mu-\nu+1}\dfrac{b}{2n+\mu+\nu+1}$;\quad
    $r\assign r+t$;\quad
    $n\assign n+1$
  }
  \Return{$r$}
\end{function}

\bibliographystyle{acmtrans} 
\bibliography{Clothoid_Spirals-refs} 

\begin{thebibliography}{}

\bibitem[\protect\citeauthoryear{Abramowitz and Stegun}{Abramowitz and
  Stegun}{1964}]{abramowitz:1964}
{\sc Abramowitz, M.} {\sc and} {\sc Stegun, I.~A.} 1964.
\newblock {\em Handbook of Mathematical Functions with Formulas, Graphs, and
  Mathematical Tables}.
\newblock Number~55 in National Bureau of Standards Applied Mathematics Series.
  U.S. Government Printing Office, Washington, D.C.

\bibitem[\protect\citeauthoryear{Backeljauw and Cuyt}{Backeljauw and
  Cuyt}{2009}]{Backeljauw:2009}
{\sc Backeljauw, F.} {\sc and} {\sc Cuyt, A.} 2009.
\newblock Algorithm 895: A continued fractions package for special functions.
\newblock {\em ACM Trans. Math. Softw.\/}~{\em 36,\/}~3 (July), 15:1--15:20.

\bibitem[\protect\citeauthoryear{Baran, Lehtinen, and Popovi{\'{c}}}{Baran
  et~al\mbox{.}}{2010}]{Baran:2010}
{\sc Baran, I.}, {\sc Lehtinen, J.}, {\sc and} {\sc Popovi{\'{c}}, J.} 2010.
\newblock Sketching clothoid splines using shortest paths.
\newblock {\em Computer Graphics Forum\/}~{\em 29,\/}~2 (May), 655--664.

\bibitem[\protect\citeauthoryear{Bertolazzi, Biral, and Lio}{Bertolazzi
  et~al\mbox{.}}{2006}]{Bertolazzi2006}
{\sc Bertolazzi, E.}, {\sc Biral, F.}, {\sc and} {\sc Lio, M.~D.} 2006.
\newblock Symbolic-numeric efficient solution of optimal control problems for
  multibody systems.
\newblock {\em Journal of Computational and Applied Mathematics\/}~{\em
  185,\/}~2, 404--421.
\newblock Special Issue: International Workshop on the Technological Aspects of
  Mathematics.

\bibitem[\protect\citeauthoryear{Davis}{Davis}{1999}]{Davis:1999}
{\sc Davis, T.} {1999}.
\newblock {Total least-squares spiral curve fitting}.
\newblock {\em Journal of Surveying Engineering-Asce\/}~{\em {125},\/}~{4}
  ({NOV}), {159--176}.

\bibitem[\protect\citeauthoryear{De~Boor}{De~Boor}{1978}]{deBoor:1978}
{\sc De~Boor, C.} 1978.
\newblock {\em A Practical Guide to Splines}.
\newblock Number v. 27 in Applied Mathematical Sciences. Springer-Verlag.

\bibitem[\protect\citeauthoryear{De~Cecco, Bertolazzi, Miori, Oboe, and
  Baglivo}{De~Cecco et~al\mbox{.}}{2007}]{dececco:2007}
{\sc De~Cecco, M.}, {\sc Bertolazzi, E.}, {\sc Miori, G.}, {\sc Oboe, R.}, {\sc
  and} {\sc Baglivo, L.} 2007.
\newblock Pc-sliding for vehicles path planning and control - design and
  evaluation of robustness to parameters change and measurement uncertainty.
\newblock In {\em ICINCO-RA (2)'2007}. 11--18.

\bibitem[\protect\citeauthoryear{Farin}{Farin}{2002}]{Farin:2001}
{\sc Farin, G.} 2002.
\newblock {\em Curves and surfaces for CAGD: a practical guide\/}, 5th ed.
\newblock Morgan Kaufmann Publishers Inc., San Francisco, CA, USA.

\bibitem[\protect\citeauthoryear{Farouki and Neff}{Farouki and
  Neff}{1995}]{Farouki:1995}
{\sc Farouki, R.~T.} {\sc and} {\sc Neff, C.~A.} 1995.
\newblock Hermite interpolation by pythagorean hodograph quintics.
\newblock {\em Math. Comput.\/}~{\em 64,\/}~212 (Oct.), 1589--1609.

\bibitem[\protect\citeauthoryear{Kanayama and Hartman}{Kanayama and
  Hartman}{1989}]{Kanayama:1989}
{\sc Kanayama, Y.} {\sc and} {\sc Hartman, B.} 1989.
\newblock Smooth local path planning for autonomous vehicles.
\newblock In {\em Robotics and Automation, 1989. Proceedings., 1989 IEEE
  International Conference on}. 1265 --1270 vol.3.

\bibitem[\protect\citeauthoryear{McCrae and Singh}{McCrae and
  Singh}{2009}]{McCrae:2008}
{\sc McCrae, J.} {\sc and} {\sc Singh, K.} 2009.
\newblock Sketching piecewise clothoid curves.
\newblock {\em Computers \& Graphics\/}~{\em 33,\/}~4 (June), 452--461.

\bibitem[\protect\citeauthoryear{{Meek} and {Walton}}{{Meek} and
  {Walton}}{1992}]{Meek:1992}
{\sc {Meek}, D.~S.} {\sc and} {\sc {Walton}, D.~J.} 1992.
\newblock {Clothoid spline transition spirals}.
\newblock {\em Mathematics of Computation\/}~{\em 59}, 117--133.

\bibitem[\protect\citeauthoryear{Meek and Walton}{Meek and
  Walton}{2004}]{Meek:2004}
{\sc Meek, D.~S.} {\sc and} {\sc Walton, D.~J.} 2004.
\newblock A note on finding clothoids.
\newblock {\em J. Comput. Appl. Math.\/}~{\em 170,\/}~2 (Sept.), 433--453.

\bibitem[\protect\citeauthoryear{Meek and Walton}{Meek and
  Walton}{2009}]{Meek:2009}
{\sc Meek, D.~S.} {\sc and} {\sc Walton, D.~J.} 2009.
\newblock A two-point g1 hermite interpolating family of spirals.
\newblock {\em J. Comput. Appl. Math.\/}~{\em 223,\/}~1 (Jan.), 97--113.

\bibitem[\protect\citeauthoryear{Pavlidis}{Pavlidis}{1983}]{Pavlidis:1983}
{\sc Pavlidis, T.} 1983.
\newblock Curve fitting with conic splines.
\newblock {\em ACM Trans. Graph.\/}~{\em 2,\/}~1 (Jan.), 1--31.

\bibitem[\protect\citeauthoryear{Press, Vetterling, Teukolsky, and
  Flannery}{Press et~al\mbox{.}}{2002}]{NR:2002}
{\sc Press, W.~H.}, {\sc Vetterling, W.~T.}, {\sc Teukolsky, S.~A.}, {\sc and}
  {\sc Flannery, B.~P.} 2002.
\newblock {\em Numerical Recipes in C++: the art of scientific computing\/},
  2nd ed.
\newblock Cambridge University Press, New York, NY, USA.

\bibitem[\protect\citeauthoryear{Scheuer and Fraichard}{Scheuer and
  Fraichard}{1997}]{Scheuer:1997}
{\sc Scheuer, A.} {\sc and} {\sc Fraichard, T.} 1997.
\newblock Continuous-curvature path planning for car-like vehicles.
\newblock In {\em Intelligent Robots and Systems, 1997. IROS '97., Proceedings
  of the 1997 IEEE/RSJ International Conference on}. Vol.~2. Intelligent Robots
  and Systems. IROS '97.,, 997--1003.

\bibitem[\protect\citeauthoryear{Shin and Singh}{Shin and
  Singh}{1990}]{Shin:1990}
{\sc Shin, D.~H.} {\sc and} {\sc Singh, S.} 1990.
\newblock Path generation for robot vehicles using composite clothoid segments.
\newblock Tech. Rep. CMU-RI-TR-90-31, Robotics Institute, Pittsburgh, PA.
  December.

\bibitem[\protect\citeauthoryear{Shirley and Chang}{Shirley and
  Chang}{2003}]{Shirley:2003}
{\sc Shirley, E.~L.} {\sc and} {\sc Chang, E.~K.} 2003.
\newblock Accurate efficient evaluation of lommel functions for arbitrarily
  large arguments.
\newblock {\em Metrologia\/}~{\em 40,\/}~1, S5.

\bibitem[\protect\citeauthoryear{Smith}{Smith}{2011}]{Smith:2011}
{\sc Smith, D.~M.} 2011.
\newblock Algorithm 911: Multiple-precision exponential integral and related
  functions.
\newblock {\em ACM Trans. Math. Softw.\/}~{\em 37,\/}~4 (Feb.), 46:1--46:16.

\bibitem[\protect\citeauthoryear{Snyder}{Snyder}{1993}]{Snyder:1993}
{\sc Snyder, W.~V.} 1993.
\newblock {Algorithm 723}: {Fresnel} integrals.
\newblock {\em {ACM} Transactions on Mathematical Software\/}~{\em 19,\/}~4
  (Dec.), 452--456.

\bibitem[\protect\citeauthoryear{Telasula}{Telasula}{2005}]{Telasula:2005}
{\sc Telasula, V.} 2005.
\newblock Fresnel cosine and sine integral function.
\newblock \url{http://www.mathworks.it/matlabcentral}.

\bibitem[\protect\citeauthoryear{Thompson}{Thompson}{1997}]{Thompson:1997}
{\sc Thompson, W.~J.} 1997.
\newblock {\em Atlas for Computing Mathematical Functions: An Illustrated Guide
  for Practitioners with Programs in C and Mathematica with Cdrom\/}, 1st ed.
\newblock John Wiley \& Sons, Inc., New York, NY, USA.

\bibitem[\protect\citeauthoryear{Walton and Meek}{Walton and
  Meek}{1996}]{Walton:1996}
{\sc Walton, D.} {\sc and} {\sc Meek, D.} 1996.
\newblock A planar cubic bezier spiral.
\newblock {\em Journal of Computational and Applied Mathematics\/}~{\em
  72,\/}~1, 85--100.

\bibitem[\protect\citeauthoryear{Walton and Meek}{Walton and
  Meek}{2008}]{Walton:2008}
{\sc Walton, D.} {\sc and} {\sc Meek, D.} 2008.
\newblock An improved euler spiral algorithm for shape completion.
\newblock In {\em Computer and Robot Vision, 2008. CRV '08. Canadian Conference
  on}. IEEE, 237--244.

\bibitem[\protect\citeauthoryear{Walton and Meek}{Walton and
  Meek}{2009}]{Walton:2009}
{\sc Walton, D.} {\sc and} {\sc Meek, D.} 2009.
\newblock Interpolation with a single cornu spiral segment.
\newblock {\em Journal of Computational and Applied Mathematics\/}~{\em
  223,\/}~1, 86--96.

\bibitem[\protect\citeauthoryear{Walton and Meek}{Walton and
  Meek}{2007}]{Walton:2007}
{\sc Walton, D.~J.} {\sc and} {\sc Meek, D.~S.} 2007.
\newblock {$G^2$} curve design with a pair of pythagorean hodograph quintic
  spiral segments.
\newblock {\em Comput. Aided Geom. Des.\/}~{\em 24,\/}~5 (July), 267--285.

\bibitem[\protect\citeauthoryear{Watson}{Watson}{1944}]{Watson:1944}
{\sc Watson, G.~N.} 1944.
\newblock {\em A {T}reatise on the {T}heory of {B}essel {F}unctions}.
\newblock Cambridge University Press, Cambridge, England.

\end{thebibliography}

\end{document}